\documentclass{amsart}
\usepackage[utf8]{inputenc}
\usepackage{amsmath,calligra,mathrsfs}
\usepackage[margin=2.5cm]{geometry}
\usepackage[english]{babel}
\usepackage{csquotes}
\usepackage[utf8]{inputenc} 
\usepackage[T1]{fontenc} 
\usepackage{amssymb,amsthm,mathtools,dsfont}
\usepackage{hyperref}
\usepackage{xcolor}
\hypersetup{
    colorlinks = true,
    linkcolor={red!80},
    citecolor={green!60!black},
    urlcolor={blue!80!black},
    pdfauthor=author
}
\usepackage[final]{microtype}
\usepackage{comment,mdframed,multicol}
\usepackage{array}
\usepackage{stmaryrd}
\usepackage{caption}
\usepackage{subcaption}
\usepackage[pdftex]{graphicx}
\usepackage{tipa}
\usepackage{adjustbox}
\usepackage{mathdots}
\usepackage{leftindex}
\usepackage{braket}
\usepackage[shortlabels]{enumitem}
\numberwithin{equation}{section}

\theoremstyle{definition}
\newtheorem{defn}{Definition}[section]
\newtheorem{ex}[defn]{Example}

\theoremstyle{plain}
\newtheorem{thm}[defn]{Theorem}
\newtheorem{lem}[defn]{Lemma}
\newtheorem{cor}[defn]{Corollary}
\newtheorem{prop}[defn]{Proposition}

\newtheorem{conj}[defn]{Conjecture}
\newtheorem*{conj*}{Conjecture}
\newtheorem{thmx}{Theorem}
\newtheorem{corx}[thmx]{Corollary}

\theoremstyle{remark}
\newtheorem{rmk}[defn]{Remark}

\newcommand{\Z}{\mathbb Z}
\newcommand{\Q}{\mathbb Q}
\newcommand{\R}{\mathbb R}
\newcommand{\C}{\mathbb C}
\newcommand{\K}{\mathbb K}

\def\GL{\mathop{\rm GL}\nolimits}

\def\cA{{\mathcal A}}

\def\cC{{\mathcal C}}
\def\cD{{\mathcal D}}

\def\cH{{\mathcal H}}

\def\cO{{\mathcal O}}
\def\cP{{\mathcal P}}

\def\Eff{\mathop{\rm Eff}\nolimits}
\def\Amp{\mathop{\rm Amp}\nolimits}

\def\Nef{\mathop{\rm Nef}\nolimits}
\def\Mov{\mathop{\rm Mov}\nolimits}

\def\Pic{\mathop{\rm Pic}\nolimits}

\def\bPic{\mathop{\rm \textbf{Pic}}\nolimits}

\def\clMov{\mathop{\overline{\rm Mov}}\nolimits}

\def\Aut{\mathop{\rm Aut}\nolimits}
\def\Spec{\mathop{\rm Spec}\nolimits}
\def\psaut{\mathop{\rm PsAut}\nolimits}
\def\bsloc{\mathop{\rm Bs}\nolimits}
\def\Supp{\mathop{\rm Supp}\nolimits}
\def\etale{\mathop{\rm \acute{e}t}\nolimits}
\def\Gal{\mathop{\rm Gal}\nolimits}
\def\Brauer{\mathop{\rm Br}\nolimits}

\newcommand{\id}{\text{id}}
\newcommand{\codim}{\operatorname{codim}}
\newcommand{\Hom}{\text{Hom}}
\newcommand{\coker}{\text{coker} ~}

\newcommand{\abs}[1]{\left\lvert#1\right\rvert}
\newcommand{\enstq}[2]{\left\{#1~\middle|~#2\right\}}

\usepackage[all]{xy}
\usepackage{tikz-cd}

\usepackage[
maxbibnames=99,
maxalphanames=4,
backend=biber,
sorting=nyt,
giveninits=true,
useprefix=true
]{biblatex}
\renewbibmacro{in:}{}

\newcommand{\Addresses}{{% additional braces for segregating \footnotesize
  \bigskip
  \footnotesize

  ~~\textsc{Institut Élie Cartan de Lorraine, Université de Lorraine et CNRS, F-54000 Nancy, France.}\par\nopagebreak
  ~~~~\textit{E-mail address}: \texttt{aurelien.faucher@univ-lorraine.fr}
}}
\title[The relative cone conjecture for fibrations in varieties with well-clipped cones]{The relative movable cone conjecture for K-trivial fibrations in varieties with well-clipped movable cones}
\author{Aurelien Faucher}
\date{\today}

\begin{document}

\begin{abstract}
We prove the weak relative Kawamata–Morrison movable cone conjecture for $K$-trivial fibrations whose very general fibre is a quotient, by a finite group of automorphisms acting freely in codimension one, of a product of certain Calabi-Yau pairs whose underlying varieties have well-clipped movable cones, a notion recently introduced by Cécile Gachet. Our main result applies in particular when the fibre is a finite product of an abelian variety, smooth rational surfaces underlying klt Calabi-Yau pairs, projective irreducible holomorphic symplectic manifolds and Enriques manifolds, both of a known type. As a consequence, there are only finitely many minimal models over the base, up to isomorphism. When the relative movable cone is non-degenerate, we obtain the full relative movable cone conjecture.
\end{abstract}

\maketitle

\section{Introduction}
It is well known that the birational geometry of a variety can be studied using cones. To cite just one of the most famous theorems in this regard, the Cone Theorem states that the $K_X$-negative part of the Mori cone of a smooth projective variety $X$ over a field $F$ is locally rational polyhedral and, provided that the characteristic of $F$ is zero, any extremal $K_X$-negative ray corresponds to a contraction morphism \cite[Theorem 3.7]{Kollár_Mori_1998}. This theorem is particularly useful when the $K_X$-negative part of the Mori cone coincides with the entire cone, e.g. for Fano varieties, as in this case we then know that the cone of curves is rational polyhedral. For a projective variety $X$ with a trivial canonical divisor, the $K_X$-positive part of this cone is not, in general, as easily describable since it may well have a circular part. This is the case, for example, for abelian surfaces of Picard rank at least three, or for certain K3 surfaces. Even if we no longer have such an explicit description, the Kawamata-Morrison cone conjecture predicts the existence of a rational polyhedral fundamental domain for the cone of effective nef classes under the action of the automorphism group of $X$. Over time, the statement of this conjecture \cite{morrison96beyond} has been greatly generalised, and following \cite{kawamata97cone} and \cite{totaro10conepairs}, we state a relative version for pairs. The various notions appearing in the following statements are standard and will be recalled in Section \ref{sect_birgeom}.

\begin{defn}\label{defn_weakfund}
    Let $G$ be a group with identity element $e$ acting on a topological space $X$. A \emph{weak fundamental domain} for $X$ under the action of $G$ is a subset $\cD \subset X$ satisfying the following conditions:
    \begin{enumerate}
        \item $X = \bigcup_{g \in G} g\cD$.
        \item For each $g \in G$, either $g\cD = \cD$ or $g\cD \cap \cD^{\circ} = \emptyset$, where $\cD^\circ$ is the interior of $\cD$.
    \end{enumerate}
    If the only element $g \in G$ satisfying $g\cD = \cD$ is $e$, then $\cD$ is called a \emph{fundamental domain}.
\end{defn}

\begin{conj}[The (weak) Kawamata-Morrison Cone conjecture]\label{conj_relKM}
Let $f : (X,\Delta) \to S$ be a klt $K$-trivial fiber space.

\begin{enumerate}
\item There exists a rational polyhedral cone $\Pi$ which is a (weak) fundamental domain under the action of the group of relative automorphisms $\Aut(X/S,\Delta)$ preserving the boundary on 
\[\Nef^+(X/S) := \operatorname{Conv}\left(\Nef(X/S) \cap N^1(X/S)_\Q \right).\]

\item There exists a rational polyhedral cone $\Pi'$ which is a (weak) fundamental domain under the action of the group of relative pseudoautomorphisms $\psaut(X/S,\Delta)$ preserving the boundary on 
\[\Mov^+(X/S) := \operatorname{Conv}\left( \clMov(X/S) \cap N^1(X/S)_\Q \right).\]
\end{enumerate}
\end{conj}

\begin{defn}
	With the notations introduced above, item $(1)$ and item $(2)$ will be, respectively, referred to as \emph{the (weak) relative nef cone conjecture} and \emph{the (weak) relative movable cone conjecture} for $(X,\Delta)$ over $S$. When $S$ is the spectrum of $\C$, thoses items will simply be referred to as the \emph{absolute (weak) movable cone conjecture} and the \emph{absolute (weak) nef cone conjecture}.
\end{defn}

It is also common to encounter formulations of this conjecture for the effective nef cone $\Nef^e(X/S) := \Nef(X/S) \cap \Eff(X/S)$ and the effective (closed) movable cone $\clMov^e(X/S) := \clMov(X/S) \cap \Eff(X/S)$. For a more comprehensive overview, we refer to \cite{LOP2018ConeConjecture}. In addition to the motivation of obtaining a description of the Mori cone of $K$-trivial varieties analogous to that for Fano varieties as discussed above, this conjecture leads to profound consequences for the birational geometry of $K$-trivial varieties. In particular, it implies the finiteness of small $\Q$-factorial modifications; see, for instance, \cite{gachet2024effective}.

The conjecture in the absolute setting has been established in a number of cases, but remains open in full generality. The case of smooth surfaces is covered by the work of Sterk \cite{sterk1985finiteness} (inspired by ideas of Looijenga), Namikawa \cite{Namikawa1985} and Kawamata \cite{kawamata97cone}. The case of klt surface pairs was proved by Totaro \cite{totaro10conepairs}. Prendergast-Smith proved the cone conjecture for abelian varieties over an algebraically closed field \cite{Prendergast-Smith2012}. The works of Markman \cite{markman2011survey}, Markman–Yoshioka \cite{markman2015proof} and Amerik–Verbitsky \cite{amerik2017morrison, amerik2020collections} establish the conjecture for irreducible holomorphic symplectic manifolds (for projective $\Q$-factorial primitive symplectic varieties $X$ with $b_2(X) \ge 5$, see \cite{lehn2024morrison}, and for projective irreducible symplectic orbifolds $Y$ with $b_2(Y) = 4$, see \cite{menet2020k}). Unfortunately, despite the efforts of many mathematicians, the case of Calabi–Yau manifolds remains open even in dimension three. For recent progress, we refer to \cite{gachet2024nef}, \cite{lutz2024morrison} and the references therein. By the Beauville-Bogomolov decomposition theorem \cite{beauville1983varietes, bogomolov1974decomposition}, the last three types of varieties mentioned play a central role in the classification of $K$-trivial varieties: they form the basic building blocks from which any $K$-trivial manifold arises up to an étale cover. While the conjecture is known for at least two of these three building blocks, there is currently no general descent result along étale coverings. Nevertheless, substantial progress has been made by Gachet in \cite{gachet2025well}. To mention just one of her results, the following was one of the main motivations for this work.

\begin{thm}[Theorem 1.4 - \cite{gachet2025well}]\label{thm_cecile}
    Let $(X,\Delta)$ be a klt Calabi-Yau pair defined over $\C$ that decomposes as
    \[X = A \times \prod_{i=1}^r Y_i \times \prod_{j=1}^s S_j, ~~\Delta = \sum_{j=1}^s p_j^*\Delta_j, \]
    where $A$ is an abelian variety, each $Y_i$ is a primitive symplectic variety with canonical singularities and $b_2(Y_i) \geq 5$, and each $S_j$ is a smooth rational surface underlying a klt Calabi-Yau pair $(S_j,\Delta_j)$. Then the movable cone conjecture holds for every quotient pair $(X/G, \Delta_G)$ of $(X,\Delta)$ by a finite subgroup $G$ of $\Aut(X,\Delta)$. Moreover, there are finitely many isomorphism classes of pairs $(Y,\Delta_Y)$ obtained by small $\Q$-factorial modifications of $(X/G,\Delta_G)$, and the nef cone conjecture holds for each of them.
\end{thm}

Less is known in the relative setting. We mention the work of Kawamata in dimension three \cite{kawamata97cone}, as well as the works \cite{li2025relative, li2023relative} and \cite{moraga2024geometric}, which establish weak versions of the relative cone conjecture for fibrations in surfaces. For a fibration $X \to S$ whose very general fibre is a projective irreducible holomorphic symplectic manifold satisfying the abundance conjecture, the relative movable cone conjecture is established in \cite[Theorems 1.3 and 1.4]{horing2024relative}. In particular, the conjecture holds whenever the very general fibre belongs to one of the known deformation classes. Under similar assumptions, the relative cone conjecture was also proved in \cite{faucher2025cone} for fibrations whose very general fibre $X_s$ is a projective primitive symplectic variety with $\Q$-factorial terminal singularities and $b_2(X_s) \ge 5$; see also \cite{fu2025finiteness}. This result also follows from \cite{gachet2025well}.

The main theorem of this paper establishes the weak relative movable cone conjecture for $\Q$-factorial terminal $K$-trivial fibrations whose very general fibre is a quotient, by a finite group of automorphisms acting freely in codimension one, of a product of certain klt Calabi-Yau pairs for which the movable cone conjecture holds and whose closed movable cone is well-clipped and contained in the effective cone.
\begin{thmx}\label{thm_A}
    Let $f : (X,\Delta) \to S$ be a $K$-trivial fibration between quasi-projective normal varieties such that $(X,\Delta)$ is a $\Q$-factorial terminal pair. Assume that the very general fibre $(X_s,\Delta_s := \Delta \vert_{X_s})$ is isomorphic to $(W/G,\Delta_{W/G})$, where $W = A \times \prod_{i=1}^p Y_i \times \prod_{j=1}^q S_j$ is a product of 
    \begin{itemize}
        \item [(i)] an abelian variety $A$,
        \item [(ii)] projective primitive symplectic varieties $Y_i$ with $\Q$-factorial terminal singularities such that $b_2(Y_i) \geq 5$, $\Mov^+(Y_i) \subset \Eff(Y_i)$, and for every klt pair $(Y_i,D_i)$, any $(K_{Y_i} + D_i)$-MMP terminates with a good minimal model,
        \item [(iii)] smooth rational surfaces $S_j$ underlying klt Calabi–Yau pairs $(S_j,\Delta_j)$.
    \end{itemize}
    Moreover, $G$ is a finite subgroup of $\Aut(W,\sum_{j=1}^q p_j^*\Delta_j)$ acting freely in codimension one, and $\Delta_{W/G}$ is the unique $\mathbb{Q}$-divisor on $W/G$ such that
    \[K_W + \sum_{j=1}^q p_j^*\Delta_j = \gamma^*(K_{W/G}+\Delta_{W/G}),\]
    where $\gamma : W \to W/G$ is the quotient morphism. Then the weak relative movable cone conjecture holds for $(X,\Delta)$ over $S$, and there are only finitely many minimal models over $S$ up to isomorphism. Moreover, if $\Mov(X/S)$ is non-degenerate, then the relative movable cone conjecture holds as well.
\end{thmx}
Theorem \ref{thm_A} can be seen as a relative version of Theorem \ref{thm_cecile} of Gachet. We briefly comment on the additional assumptions, which differ from those in Gachet’s result. The termination of every MMP for klt pairs on the primitive symplectic factors is essential in our proof, as it is needed to establish the existence of good minimal models on $W$ (see Section \ref{sect_prod} for more details). Similarly, we require the action of $G$ to be free in codimension one in order to establish the same result on $W/G$ (see Section \ref{sect_proof}). These additional assumptions are satisfied in many relevant cases; for instance, we obtain the following corollary.

\begin{corx}\label{cor_B}
    Let $f : (X,\Delta) \to S$ be a $K$-trivial fibration between quasi-projective normal varieties such that $(X,\Delta)$ is a $\Q$-factorial terminal pair. Assume that the very general fibre $(X_s,\Delta_s)$ is isomorphic to $(W,\Delta_W)$, where
    \[W = A \times \prod_{i=1}^p Y_i \times \prod_{k=1}^\ell Z_k \times \prod_{j=1}^q S_j,\]
    $A$ is an abelian variety, each $Y_i$ is a projective irreducible holomorphic symplectic manifold of a known type, each $Z_k$ is an Enriques manifold of a known type, each $S_j$ is a smooth projective rational surface underlying a klt Calabi-Yau pair $(S_j,\Delta_j)$, and $\Delta_W := \sum_{j=1}^q p_j^*\Delta_j$. Then the weak relative movable cone conjecture holds for $(X,\Delta)$ over $S$, and there are only finitely many minimal models over $S$ up to isomorphism. Moreover, if $\Mov(X/S)$ is nondegenerate, then the relative movable cone conjecture also holds.
\end{corx}
Enriques manifolds, introduced and studied independently by Boissière-Nieper-Wißkirchen-Sarti \cite{boissiere2011higher} and Oguiso-Schröer \cite{oguiso2011enriques, oguiso2011periods}, naturally fit into the framework of Theorem \ref{thm_A}, since they are finite étale quotients of projective irreducible holomorphic symplectic manifolds. For simplicity, we say that an Enriques manifold is of known type if its universal cover is deformation equivalent to one of the known irreducible holomorphic symplectic types. As a consequence of our results, we recover, to the best of our knowledge, all known results (or special cases thereof) concerning the relative movable cone conjecture, including \cite[Corollary 1.9]{li2023relative}, \cite[Theorem 1.3]{horing2024relative}, \cite[Theorem 1.2]{moraga2024geometric}, \cite[Corollary 8.5]{engel2025boundedness}, and \cite[Theorem B]{faucher2025cone}. The proof relies on the following general result, which is inspired by and strengthens \cite[Theorems 1.4 and 1.8]{li2023relative}. For the definition of effective klt perturbations, see Definition \ref{def_klteffpert}.

\begin{thmx}\label{thm_C}
    Let $f : (X,\Delta) \to S$ be a $K$-trivial fibration between quasi-projective normal varieties such that $(X,\Delta)$ is a $\Q$-factorial terminal pair. Let $X_\eta$ be the generic fibre of $f$. Assume the following conditions.
    \begin{enumerate}
        \item For a very general point $s \in S$, good minimal models exist for effective klt perturbations of $(X_s, \Delta \vert_{X_s})$.
        \item There exists a rational polyhedral cone $P_\eta \subset \Eff(X_\eta)$ such that $\Mov(X_\eta) \subset \psaut(X_\eta, \Delta_\eta) \cdot P_\eta$.
    \end{enumerate}
    Then the weak relative movable cone conjecture holds for $X/S$ and there are only finitely many minimal models over $S$ up to isomorphism. Moreover, if $\Mov(X/S)$ is non-degenerate, then the relative movable cone conjecture holds as well.
\end{thmx}
We emphasize that Theorem \ref{thm_C} and its proof originate in the works \cite{li2023relative, li2025relative}. Our contribution consists in relaxing the assumptions in \cite[Theorems 1.4 and 1.8]{li2023relative}; more precisely, whereas Li and Li-Zhao require the existence of good minimal models for effective klt pairs in relative dimension $\dim(X/S)$, we weaken this assumption by replacing it with item (1) of Theorem \ref{thm_C}. The construction of the weak rational polyhedral fundamental domain is explicit enough to show that, when the relative movable cone is non-degenerate, the movable cone conjecture (and not only its weak form) holds. The proof of Theorem \ref{thm_A} then proceeds by showing that good minimal models exist for effective klt perturbations of the very general fiber, and by applying \cite{gachet2025well} to obtain the movable cone conjecture for $\Mov^+(X_\eta)$. This relies on the notion of a well-clipped cone introduced by Gachet, which we recall in Section \ref{sect_prelconv}, together with the fact that the (closed) movable cones of the following varieties are well-clipped: abelian varieties, projective irreducible holomorphic symplectic manifolds, projective terminal $\Q$-factorial primitive symplectic varieties $W$ with $b_2(W)\ge 5$, and smooth projective surfaces underlying klt Calbi-Yau pairs (see \cite[Example 3.6]{gachet2025well} and Example \ref{ex_wellclipped}). The results of \cite{gachet2025well} play a crucial role throughout this paper. Once this is established, we show that the movable cone of the generic fibre is contained in the effective cone, and we then apply Theorem \ref{thm_C}. 

It is important to note that the non-degeneracy assumption is essential. Indeed, building on the example \cite[Example 6.8]{reid1983minimal} which goes back to \cite{shepherd1980some}, Kawamata shows in \cite[Example 3.8.(2)]{kawamata97cone} that there exists an elliptic fibration $f : X \to S$ of relative Picard rank two with $X$ and $S$ smooth such that $\clMov(X/S)$ is a degenerate cone. Moreover, in this example, Kawamata shows that
\[\Mov(X/S) = \clMov^e(X/S) \subsetneq \Mov^+(X/S).\]
Therefore, the existence of a rational polyhedral fundamental domain for $\Mov^+(X/S)$ under the action of $\psaut(X/S,\Delta)$ does not necessarily imply the same for $\clMov^e(X/S)$.

\subsection*{Organization}
In Section \ref{sect_prel}, we recall classical definitions from convex and birational geometry. We also take this opportunity to recall results from \cite{li2023relative, li2025relative} on the construction of (weak) rational polyhedral fundamental domains, which will be indispensable for proving our theorems. We further recall the definition of a well-clipped cone introduced by Gachet in \cite{gachet2025well}. In Section \ref{sect_groupscheme}, we present Li's work \cite[Section 4]{li2023relative} on group schemes associated with $K$-trivial fibrations. In Section \ref{sect_gentorel}, we show that the results of \cite{li2023relative} remain valid under a weakening of one of the hypotheses. This section contains technical results needed to establish the relative movable cone conjecture, as well as the proof of Theorem \ref{thm_C}. In Section \ref{sect_prod}, we determine conditions under which a product of varieties admits good minimal models. These results are used in the proofs of Theorem \ref{thm_A} and Corollary \ref{cor_B} which are given in Section \ref{sect_proof}.

\subsection*{Acknowledgements}
This paper is part of my doctoral work. I am deeply grateful to my PhD advisor, Gianluca Pacienza, for introducing me to this problem and for his guidance throughout this project. I also thank him for his careful reading of earlier versions of this paper and for his many helpful comments and suggestions. I would like to thank Francesco A. Denisi and Zhixin Xie for answering my questions and for their invaluable help in establishing the results of Section \ref{sect_prod}. I am also grateful to the authors of \cite{denisi2024mmp} for kindly sharing a preliminary version of their revised paper. Finally, I thank Cécile Gachet for answering my questions about her work \cite{gachet2025well} and for the discussions we had which were extremely beneficial to me.

\subsection*{Conventions}
A \emph{variety} is a geometrically integral, separated scheme of finite type over a field $F$. A very general fibre of a morphism between complex varieties always means a fibre over a very general closed point. Similarly, except for the generic point of a scheme, all points are assumed to be closed unless otherwise specified.

\section{Preliminaries}\label{sect_prel}

\subsection{Convex geometry}\label{sect_prelconv}
In the following, we denote by $V_\Z$ an integral lattice, i.e. a finitely generated free abelian group. For $\K = \Q$ or $\R$, we set $V_\K := V_\Z \otimes_\Z \K$. For simplicity, we denote by $V$ the vector space $V_\R$. If $S$ is a subset of a topological space $X$, we denote the interior of $S$ by $S^\circ$.

\begin{defn}\label{def_cone}
    A \emph{cone} $C \subset V$ is a subset of $V$ stable by multiplication by positive scalars. A convex cone $C \subset V$ is
    \begin{itemize}
        \item \emph{full-dimensional} if the linear span of $C$ is equal to $V$.
        \item \emph{degenerate} if its closure $\overline{C}$ in $V$ contains a non-trivial subvector space of $V$.
    \end{itemize}
    For a degenerate convex cone $C$, the \emph{maximal vector subspace of $\overline{C}$} is defined as the intersection
    \[W(\overline{C}) := \overline{C} \cap \left( -\overline{C} \right).\]
    We say that $W(\overline{C})$ is \emph{defined over $\Q$} if $W(\overline{C}) = W(\overline{C})_\Q \otimes_\Q \R$, where $W(\overline{C})_\Q = W(\overline{C}) \cap V_\Q$. When the context is sufficiently clear, we will simply write $W$ to denote $W(\overline{C})$.
\end{defn}

\begin{defn}
    For a subset $S \subset V$, we denote by Cone$(S)$ the closed convex cone spanned by $S$ in $V$ and by Conv$(S)$ the convex hull of $S$ in $V$. A closed convex cone $C \subset V$ is \emph{polyhedral} (respectively, \emph{rational polyhedral}) if there exists a finite set $S \subset V$ (respectively, a finite set $S \subset V_\Q$) such that $C =$ Cone$(S)$.
\end{defn}

\begin{defn}
    A \emph{(rational) polytope} $Q \subset V$ is a subset which is the convex hull of finitely many (rational) points of $V$. In particular, a polytope is always closed.
\end{defn}
In the following, $\Gamma$ will denote a group acting (on the left) linearly on $V$. This is equivalent to the existence of a group morphism $\rho : \Gamma \to \GL(V)$. With regard to notations, for any $x \in V$ and any $\gamma \in \Gamma$, we will write $\gamma x$ to denote the element $\rho(g)(x)$. We say that $\Gamma$ acts \emph{faithfully} on $V$ if $\rho$ is injective. The following definitions are taken from \cite[Section 4]{looijenga2014discrete} and \cite[Section 3]{li2025relative}.

\begin{defn}\label{def_poltype}
    Let $C \subset V$ be a full-dimensional open convex cone, and $\Gamma$ a group acting faithfully and linearly on $V$. Assume that the action leaves $C$ and some lattice in $V_\Q$ invariant. Set 
    \[C^+ := \operatorname{Conv}(\overline{C} \cap V_\Q).\]
    We say that $(C^+, \Gamma)$ is of \emph{polyhedral type} if there exists a polyhedral cone $\Pi \subset C^+$ such that 
    \[C \subset \Gamma \cdot \Pi,\]
    where $\Gamma \cdot \Pi = \enstq{\gamma x}{\gamma \in \Gamma, x \in \Pi}$.
\end{defn}
\begin{rmk}
    With the notations and assumptions as above, the hypothesis on $C$ being open is not necessary. Indeed, since $\overline{C^\circ} = \overline{C}$, it follows that $C^+ = (C^\circ)^+$ and we can work with the interior of $C$.
\end{rmk}

\begin{prop}[Proposition-Definition 4.1 - \cite{looijenga2014discrete}, Proposition 3.2 - \cite{li2023relative}]\label{prop_poltype}
    Let $C \subset V$ be a full-dimensional non-degenerate open convex cone, and $\Gamma$ a group acting faithfully and linearly on $V$. Assume that the action leaves $C$ and some lattice in $V_\Q$ invariant. The following properties are equivalent.
    \begin{enumerate}
        \item There exists a polyhedral cone $\Pi \subset C^+$ such that $C^+ = \Gamma \cdot \Pi$.
        \item There exists a polyhedral cone $\Pi \subset C^+$ such that $C \subset \Gamma \cdot \Pi$, i.e. $(C^+,\Gamma)$ is of polyhedral type.
        \item There exists a polyhedral cone $\Pi \subset C^+$ such that $C \cap V_\Q \subset \Gamma \cdot \Pi$.
        \item For every $\Gamma$-invariant lattice $L \subset V_\Q$, $\Gamma$ has finitely many orbits in the set of extreme points of $\operatorname{Conv}(C \cap L)$.
    \end{enumerate}
    In case (2), we actually have $C^+ = \Gamma \cdot \Pi$. 
\end{prop}

\begin{lem}[Theorem 3.8, Application 4.14 - \cite{looijenga2014discrete}, Lemma 3.5 - \cite{li2025relative}]
    Let $C \subset V$ be a full-dimensional open convex cone, and $\Gamma$ a group acting faithfully and linearly on $V$, leaving $C$ and some lattice in $V_\Q$ invariant. Assume that $C$ is non-degenerate and that $(C^+, \Gamma)$ is of polyhedral type. Then, under the action of $\Gamma$, $C^+$ admits a rational polyhedral fundamental domain.
\end{lem}
For the moment, the existence of fundamental domains is known under the assumption that the cone is non-degenerate. The work of \cite{li2023relative, li2025relative}, and more precisely \cite[Proposition 3.6]{li2023relative} and \cite[Proposition 3.8]{li2025relative}, allow these results to be generalized to ensure the existence of weak fundamental domains for degenerate open convex cones. Using the notation of Definition \ref{def_cone}, let $C \subset V$ be an open convex cone and denote by $W$ the maximal vector subspace contained in $\overline{C}$. When $W$ is defined over $\Q$, the vector space $V/W$ inherits a natural lattice structure by the isomorphism
\[V/W \simeq V_\Q / W_\Q \otimes_\Q \R.\]
Let $p : V \to V/W$ be the quotient map. For every subset $S \subset V$, we denote by $\widetilde{S}$ the image $p(S)$ of $S$ under $p$. By maximality of $W$, $\Gamma$ necessarily preserves $W$ and thus induces an action $\tilde{\rho} : \Gamma \to \GL(V/W)$ that leaves $\widetilde{C}$ invariant.

\begin{prop}[Proposition 3.6 - \cite{li2023relative}]\label{prop_li3.6}
    Let $C \subset V$ be a full-dimensional open convex cone, and $\Gamma$ a group acting linearly on $V$ leaving $C$ and some lattice in $V_\Q$ invariant. Assume that the maximal vector subspace $W \subset \overline{C}$ is defined over $\Q$. Denote by $\widetilde{\Gamma}$ the subgroup $\tilde{\rho}(\Gamma)$ of $\GL(V/W)$. If $(\widetilde{C}^+, \widetilde{\Gamma})$ is of polyhedral type, then there is a rational polyhedral cone $\Pi \subset C^+$ such that $\Gamma \cdot \Pi = C^+$, and for each $\gamma \in \Gamma$, either $\gamma \Pi \cap \Pi^\circ = \emptyset$ or $\gamma \Pi = \Pi$. Moreover,
    \[\enstq{\gamma \in \Gamma}{\gamma \Pi = \Pi} = \enstq{\gamma \in \Gamma}{\gamma \text{ acts trivially on } V/W}.\]
\end{prop}
We conclude this section by recalling the definition of a well-clipped cone introduced by Gachet in her paper \cite{gachet2025well}. To this end, we also rely on the notations and definitions introduced in \cite[Section 2.A]{gachet2025well}.

\begin{defn}
    Let $\cA \subset V$ be an open non-degenerate cone. We say that
    \begin{enumerate}
        \item $\cA$ is \emph{self-dual} if there exists a positive definite quadratic form $\operatorname{tr} : V^{\otimes 2} \to \R$ that induces an identification of $\cA$ with its dual cone.
        
        \item $\cA$ is \emph{homogeneous} if the group of linear automorphisms of $V$ that preserve $\cA$ acts transitively on $\cA$.

        \item $\cA$ is \emph{$\R$-indecomposable} if, for any decomposition $V = V_1 \oplus V_2$ into $\R$-linear subspaces with both $\cA_i := \cA \cap V_i \neq \emptyset$, we have a strict inclusion $\cA_1 + \cA_2 \subsetneq \cA$.
    \end{enumerate}
    We denote by $\Aut(\cA)$ the subgroup of $\GL(V)$ consisting of linear transformations that preserve the cone $\cA$, and by $\Aut(\cA,V_\Z)$ its subgroup that also preserves the lattice $V_\Z$.
\end{defn}

\begin{ex}[Example 2.2 - \cite{gachet2025well}]
    Let $n \geq 3$ and $q$ the quadratic form on $\R^n$ of signature $(1,n-1)$. Fix $h \in \R^n$ such that $q(h) > 0$. The non-degenerate open cone
    \[\cH_n := \enstq{v \in \R^n}{q(v) > 0 ~\text{and}~ q(v,h) > 0}\]
    is an $\R$-indecomposable self-dual homogeneous cone. A cone that identifies up to isomorphism with $\cH_n$ is said to be \emph{of hyperbolic type}, or \emph{hyperbolic}.
\end{ex}

\begin{lem}[Lemma 2.6 - \cite{gachet2025well}]\label{lem_gachet26}
    Let $\cA$ be a self-dual homogeneous cone in $V$. Then, there exists a positive definite $V_\Z$-integral quadratic form $\operatorname{tr}$ on $V$ that is $\Aut(\cA,V_\Z)$-invariant and with respect to which $\cA$ is self-dual.
\end{lem}

\begin{defn}[Well-clipped cone, Definition 3.1 - \cite{gachet2025well}]
Let $\cA \subset V$ be a self-dual homogeneous cone. A full-dimensional convex cone $\cC \subset V$ is said to be \emph{well-clipped in $\cA$} if there exists a family of hyperplanes $(H_i)_{i \in I}$ such that
\[\cC^\circ = \cA \cap \bigcap_{i \in I} H_{i,+},\]
where $H_{i,+}$ is one of the two connected components of $V \backslash H_i$, and the following conditions are satisfied:

\begin{enumerate}
\item Decompose $\cA = \bigoplus_{j\in J} \cA_j$ into a direct sum of $\R$-indecomposable
summands. Every hyperplane $H_i$ is of the form
\[H_i = H_i \cap \operatorname{Span}_\R \cA_{j(i)} \oplus \bigoplus_{k \neq j(i)} \operatorname{Span}_\R(\cA_k),
\]
and the remaining cone $\mathcal A_{j(i)}$ is of hyperbolic type and is defined over $V_{\mathbb Q}$.
Moreover, fix an $\Aut(\cA,V_\Z)$-invariant and $V_\Z$-integral quadratic form $q$
which is a direct sum of hyperbolic forms on the linear spans of the $\left(\cA_{j(i)}\right)_{i \in I}$ and of a positive definite form as in Lemma \ref{lem_gachet26} on the other summands' spans.

\item For every $i\in I$, the $q$-orthogonal reflection $\sigma_i$ with fixed hyperplane $H_i$
preserves the lattice $V_\Z$.

\item For every $i,k \in I$ such that $H_i \neq H_k$ and for any vectors $e_i, e_k$ perpendicular to $H_i, H_k$ with negative $q$-squares satisfy $q(e_i,e_k)\ge 0$.
\end{enumerate}

We call a cone \emph{well-clipped} if it is well-clipped in some self-dual homogeneous cone
$\cA \subset V$.
\end{defn}

\begin{ex}\label{ex_wellclipped}
    To illustrate this concept, Gachet provided a list of interesting examples of well-clipped cones \cite[Example 3.6]{gachet2025well}, to which we refer for further details. For the purposes of this paper, we restrict ourselves to those that are relevant to our setting.
\begin{enumerate}
    \item The closed movable cone of a projective irreducible holomorphic symplectic manifold is well-clipped.
    \item The closed movable cone of a projective primitive symplectic variety with terminal $\Q$-factorial singularities is well-clipped.
    \item The closed movable cone of an abelian variety is well-clipped.
    \item The nef cone (which coincides with the closed movable cone) of a smooth projective surface $S$ underlying a klt pair $(S,\Delta)$ with $K_S + \Delta \equiv 0$ is well clipped.
\end{enumerate}
\end{ex}

\subsection{Birational geometry}\label{sect_birgeom}
Although some of the definitions below make sense in a more general sense, we will only be working over the field $\C$ of complex numbers. Using the terminology of \cite[Section 2]{Kollár_Mori_1998}, we recall some definitions from birational geometry in order to introduce the objects needed later. For extensions to the case of pairs with an $\R$-divisor as boundary, we also refer to \cite{fujino2017foundations}. Throughout this paper, a morphism $f : X \to S$ between varieties will be called a \emph{fibration} if it is proper, surjective, and has connected fibres.

\begin{defn}
    Let $f : X \to S$ be a fibration between $\Q$-factorial normal varieties. Two divisors $D$ and $D'$ on $X$ are \emph{$f$-linearly equivalent}, written $D \sim_{f} D'$ or $D \sim_S D'$, if there exists a divisor $B$ on $S$ such that
    \[D - D' \sim f^*B\]
    We denote by $\Pic(X/S)$ the relative Picard group (of $f$), and for $\K = \Q$ or $\R$, we set
    \[\Pic(X/S)_{\K} = \Pic(X/S) \otimes_{\Z} \K.\]
\end{defn}

\begin{defn}
    Let $f : X \to S$ be a fibration between $\Q$-factorial normal varieties. A $\Q$-divisor $D$ on $X$ is
    \begin{itemize}
        \item \emph{$f$-effective}, or \emph{effective over $S$}, if for some sufficiently divisible natural number $m$, the coherent $\cO_S$-module $f_*\cO_X(mD)$ has positive rank.
        \item \emph{$f$-movable}, or \emph{movable over $S$}, if for some sufficiently divisible natural number $m$, we have 
        \[\codim(\Supp(\coker \operatorname{ev})) \geq 2,\] 
        where $\operatorname{ev} : f^*f_*\cO_X(mD) \to \cO_X(mD)$ is the natural evaluation map.
        \item \emph{$f$-semiample}, or \emph{semiample over $S$}, if for some sufficiently divisible natural number $m$, the evaluation map ev is surjective.
    \end{itemize}
    An $\R$-divisor $D$ is $f$-effective (respectively, $f$-movable or $f$-semiample) if it can be written as $D = \textstyle\sum_i a_i D_i$ where the $D_i$ are $f$-effective (respectively, $f$-movable or $f$-semiample) $\Q$-divisors, and the $a_i$ are non-negative real numbers.
\end{defn}
Note that a $\Q$-divisor $D$ on $X$ is $f$-effective if for some sufficiently divisible natural number $m$, the restriction $(mD)\vert_{X_\eta}$ of $mD$ on the generic fibre is effective, where $X_{\eta}$ is the generic fibre of $f$.

\begin{lem}[Lemma 2.2 - \cite{horing2024relative}]\label{lem_effrel}
    Let $f : X \to S$ be a fibration between normal $\Q$-factorial varieties and $D$ an $\R$-divisor on $X$. If $D$ is $f$-effective, then there exists an effective $\R$-divisor $D'$ on $X$ such that 
    \[D \sim_{\R,f} D'.\]
\end{lem}

\begin{defn}
    Let $f : X \to S$ be a fibration between $\Q$-factorial normal varieties. Two divisors $D$ and $D'$ on $X$ are \emph{$f$-numerically equivalent}, written $D \equiv_{f} D'$ or $D \equiv_S D'$, if $D \cdot C = D' \cdot C$ for every curve $C \subset X$ such that $f(C)$ is a point. We denote by $N^1(X/S)$ the relative Néron-Severi group (of $f$), and for $\K = \Q$ or $\R$, we set
    \[N^1(X/S)_\K = N^1(X/S) \otimes_{\Z} \K.\]
\end{defn}

\begin{defn}\label{def_ktrivfibrespace}
Let $X$ be a normal complex variety and let $\Delta$ be an $\R$-divisor on $X$. Then $(X,\Delta)$ is called a \emph{(log) pair} if $K_X + \Delta$ is $\R$-Cartier. The divisor $\Delta$ is called a \emph{boundary divisor}, or simply a \emph{boundary}. A pair $(X,\Delta)$ is a Calabi-Yau pair if $K_X + \Delta \equiv 0$. A \emph{$K$-trivial fiber space} is a $\Q$-factorial pair $(X,\Delta)$ endowed with a fibration $f : X \to S$ over a quasi-projective variety $S$ such that $K_X + \Delta \equiv_{S} 0$. If $(X,\Delta)$ is klt (respectively, terminal), we say that $f$ is a klt (respectively, terminal) $K$-trivial fiber space.
\end{defn}

\begin{defn}
    Let $f : X \to S$ be a fibration between $\Q$-factorial normal varieties. Inside the vector space $N^1(X/S)_\R$, we define the following convex cones.
    \begin{itemize}
        \item The \emph{relative effective cone} $\Eff(X/S)$, generated by $f$-effective divisors.
        \item The \emph{relative movable cone} $\Mov(X/S)$, generated by $f$-movable divisors, and its closure $\clMov(X/S)$. We set 
        \[\clMov^e(X/S) := \clMov(X/S) \cap \Eff(X/S),\]
        and we define the \emph{modified movable cone} $\Mov^+(X/S)$ as the convex hull of $\clMov(X/S) \cap N^1(X/S)_\Q$ in $N^1(X/S)_\R$.
        \item The \emph{relative ample cone} $\Amp(X/S)$, generated by $f$-ample divisors, and its closure $\Nef(X/S)$, called the \emph{relative nef cone}. We set 
        \[\Nef^e(X/S) := \Nef(X/S) \cap \Eff(X/S),\]
        and we define the \emph{modified nef cone} $\Nef^+(X/S)$ as the convex hull of $\Nef(X/S) \cap N^1(X/S)_\Q$ in $N^1(X/S)_\R$. 
    \end{itemize}
\end{defn}

\begin{defn}
    A birational map $g : X \dashrightarrow Y$ between normal projective varieties is a \emph{contraction} if its inverse $g^{-1}$ does not contract a divisor. In addition, if for two fibrations $f_X : X \to S$ and $f_Y : Y \to S$, we have $f_X = f_Y \circ g$ on the locus where $g$ is defined, we say that $g$ is a \emph{birational contraction over $S$}.
\end{defn}

\begin{defn}
    Let $(X_1, \Delta_1)$ and $(X_2, \Delta_2)$ be two $\Q$-factorial klt pairs. We say that a birational map $\mu : X_1 \dashrightarrow X_2$ is a \emph{small $\Q$-factorial modification of $(X_1, \Delta_1)$ over $S$} if $\mu$ is an isomorphism in codimension $1$ over $S$ and if $\mu_*\Delta_1 = \Delta_2$, where $\mu_*\Delta_1$ is the birational transform of $\Delta_1$ by $\mu$.
\end{defn}

\begin{defn}
    Let $f : (X, \Delta) \to S$ be a fibration between normal quasi-projective varieties. A \emph{pseudoautomorphism of $(X,\Delta)$} is a birational map $\mu : X \dashrightarrow X$ over $S$ that is an isomorphism in codimension $1$ and that satisfies $\mu_*\Delta = \Delta$. The set of pseudoautomorphisms of $(X,\Delta)$ forms a group denoted by $\psaut(X/S,\Delta)$. When $\Delta = 0$, we simply write $\psaut(X/S)$. The set of automorphisms of $(X,\Delta)$ preserving the boundary $\Delta$ under birational transform will be denoted by $\Aut(X/S,\Delta)$.
\end{defn}

\begin{defn}\label{def_minmodel}
    Let $X \to S$ be a fibration with $X$ a normal projective variety. Let $\Delta$ be a divisor on $X$ such that $(X,\Delta)$ is a log-canonical pair. For every birational contraction $\phi : X \dashrightarrow Y$ of normal projective varieties over $S$, we set $\Delta_Y := \phi_* \Delta$. A \emph{minimal model} $(Y,\Delta_Y)$ of $(X,\Delta)$ over $S$ is a pair associated to a birational contraction $\phi : (X,\Delta) \dashrightarrow (Y,\Delta_Y)$ such that
    \begin{itemize}
        \item $Y$ is $\Q$-factorial,
        \item $K_Y + \Delta_Y$ is nef over $S$ (where $K_Y$ is the canonical divisor of $Y$), and
        \item $a(E,X,\Delta) > a(E,Y,\Delta_Y)$ for all $\phi$-exceptional divisors.
    \end{itemize}
    We will say that a minimal model $(Y,\Delta_Y)$ over $S$ is \emph{a good minimal model} if, in addition, $K_Y + \Delta_Y$ is semiample over $S$. A minimal model $(Y, \Delta_Y)$ together with the birational contraction $\phi$ is called a \emph{marked minimal model} of $(X, \Delta)$ over $S$.
\end{defn}

\begin{rmk}\label{rmk_allgood}
    With the notations as above, note that if $(X,\Delta)$ has a good minimal model over $S$, then any minimal model of $(X,\Delta)$ over $S$ is good by \cite[Remark 2.7]{birkar2012existence}.
\end{rmk}

\begin{defn}\label{def_klteffpert}
    Let $(X,\Delta)$ be a $\Q$-factorial klt pair. We say that
    \begin{itemize}
        \item [(i)] \emph{(good) minimal models exist for effective klt perturbations of $(X,\Delta)$} if, for every $\R$-divisor $D$ on $X$ such that $(X,\Delta + D)$ is klt and $\kappa(X,K_X+\Delta + D) \geq 0$, the pair $(X,\Delta + D)$ has a (good) minimal model;
        \item [(ii)] \emph{(good) minimal models exist for effective klt pairs on $X$} if (good) minimal models exist for effective klt perturbations of $(X,0)$. 
    \end{itemize}
     Here, $\kappa(X,B)$ denotes the Iitaka dimension of an $\R$-Cartier divisor $B$ on $X$.
\end{defn}
It will be useful later on to use the following theorem, originally due to \cite{hacon2013existence} and extended to the case of $\R$-divisors by \cite{li2023relative}.

\begin{thm}[Theorem 2.12 - \cite{hacon2013existence}, Theorem 2.23 - \cite{li2023relative}]\label{thm_hxli}
    Let $f : X \to S$ be a surjective projective morphism and let $(X,\Delta)$ be a klt pair. Assume that for a very general point $s \in S$, the pair $(X_s, \Delta_s)$ has a good minimal model. Then $(X,\Delta)$ has a good minimal model over $S$. 
\end{thm}

\section{Group schemes associated to K-trivial fiber spaces}\label{sect_groupscheme}
In this section, we recall some facts concerning both the automorphism functor and the Picard functor. These results will be useful to us for $K$-trivial fibrations thanks to the work of \cite[Section 4]{li2023relative}, which we will recall. For a more complete overview, see also \cite{fantechi2005fundamental}. In the following, if $S$ is a scheme and if $X,T$ are two $S$-schemes, we denote by $X_T$ the $T$-scheme $X \times_S T$. If $F'/F$ is a field extension, for $S = \Spec(F)$ and $T = \Spec(F')$, we will write more simply $X_{F'}$ or $X \otimes_F F'$ to denote $X_T$. Following the notation of \cite{li2023relative}, for a scheme $X$ and a field $F$, we write $[x] \in X(F)$ for an $F$-valued point of $X$ and $x \in X$ for its image (i.e. the underlying point of $X$). Although we also use square brackets to denote numerical equivalence classes of divisors, the context should prevent any confusion.

\begin{defn}
    Let $X$ be a proper scheme over a field $F$. Denote by (Sch/$F$) the category of schemes over $F$. The \emph{automorphism functor} of $X$ is the contravariant functor
    \[\mathfrak{A}ut_{X/F} : (\text{Sch}/F ) \to (\text{Grp}), ~~\mathfrak{A}ut_{X/F}(T) := \Aut(X_T/T)\]
    where $\Aut(X_T/T)$ is the group of all automorphisms of $X_T$ over $T$. If $i : Z \to X$ is a closed subscheme of $X$, we define the subfunctor $\mathfrak{A}ut_{X/F, Z}$ of $\mathfrak{A}ut_{X/F}$ by
    \[\mathfrak{A}ut_{X/F, Z}(T) := \enstq{g \in \mathfrak{A}ut_{X/F}(T)}{g \circ i_T : Z_T \to X_T ~\text{factors through}~ i_T}.\]
\end{defn}

The first statement of the following theorem is \cite[Theorem 3.7]{matsumura1967representability}, and the second one follows from \cite[Theorem II.1.3.6]{demazure70groupes}.
\begin{thm}\label{thm_autexist}
    Let $X$ be a proper scheme over a field $F$ and let $Z \to X$ be a closed subscheme of $X$. Then, the following statements hold.
    \begin{itemize}
        \item The functor $\mathfrak{A}ut_{X/F}$ is representable by a group scheme \textbf{Aut}$_F(X)$ locally of finite type over $F$.
        \item The functor $\mathfrak{A}ut_{X/F, Z}$ is representable by a group scheme \textbf{Aut}$_F(X, Z)$ which is a closed subgroup scheme of \textbf{Aut}$_F(X)$.
    \end{itemize}
\end{thm}
Note that the $F$-points of \textbf{Aut}$_F(X, Z)$ are exactly the automorphisms of $X$ over $F$ preserving $Z \to X$. When $(X,\Delta)$ is a pair, we will denote by \textbf{Aut}$_F(X, \Delta)$ the group scheme \textbf{Aut}$_F(X, \Supp \Delta)$, where $\Supp \Delta$ is endowed with its reduced structure. When the context is sufficiently clear, we will drop the index $F$ and simply write \textbf{Aut}$(X)$ and \textbf{Aut}$(X,\Delta)$.

\begin{defn}
    Let $S$ be a scheme and $X$ an $S$-scheme. The relative Picard functor $\cP ic_{X/S}$ is the contravariant functor defined by
    \[\cP ic_{X/S} : (\text{Sch}/S ) \to (\text{Ab}), ~~\cP ic_{X/S}(T) := \Pic(X_T)/p_T^*\Pic(T),\]
    where $p_T : X_T \to T$ is the projection. We denote its associated sheaf in the big étale topology by $\cP ic_{(X/S)(\etale)}$.
\end{defn}

\begin{thm}[Theorem 9.4.8 - \cite{fantechi2005fundamental}]
    Let $f : X \to S$ be a morphism of schemes. Assume that $f$ is projective Zariski locally over $S$, flat and has integral geometric fibres. Then, the sheaf $\cP ic_{(X/S)(\etale)}$ is representable by a separated group scheme $\bPic(X/S)$ locally of finite type over $S$.
\end{thm}

\begin{rmk}\label{rmk_incpic}
    The situation that will interest us most in what follows is when S is the spectrum of a field $F$. In this case, there is a natural inclusion $\Pic(X) \hookrightarrow \bPic(X/F)(F)$, which is in general not an isomorphism. The failure of surjectivity is measured by the Brauer group $\Brauer(F)$ of $F$ \cite[Corollary 2.5.9]{colliot2021brauer}. Nevertheless, when $F$ is algebraically closed, then this map is indeed an isomorphism.
\end{rmk}
Under the appropriate assumptions that guarantee the existence of both the Picard scheme and the automorphism scheme, the following result is a direct application of \cite[Lemma 9.5.1]{fantechi2005fundamental}.

\begin{prop}\label{prop_concomp}
    Let $F$ be a field, $X$ a projective variety over $F$ and $Z$ a closed subscheme of $X$. Then the connected components
    \begin{itemize}
        \item \textbf{Aut}$^{0}(X)$ of the identity element of the $F$-group scheme \textbf{Aut}$(X)$,
        \item \textbf{Aut}$^{0}(X,Z)$ of the identity element of the $F$-group scheme \textbf{Aut}$(X,Z)$,
        \item $\bPic^0(X/F)$ of the identity element of the $F$-group scheme $\bPic(X/F)$
    \end{itemize}
    are open and closed group subschemes of finite type over $F$. Moreover, they are geometrically irreducible and their formations commute with extending $F$. More precisely, if $F'/F$ is a field extension, then
    \[\bPic^0(X/F) \otimes_F F' \simeq \bPic^0(X_{F'}/F'), ~\textbf{Aut}^0(X,Z) \otimes_F F' \simeq \textbf{Aut}^0(X_{F'},Z_{F'}) ~\text{and}~ \textbf{Aut}^0(X) \otimes_F F' \simeq \textbf{Aut}^0(X_{F'}).\]
\end{prop}

\begin{defn}
    Using the same notation and assumptions as above, we define
    \[\operatorname{Aut}^0(X,\Delta) := \textbf{Aut}^0(X,Z)(F).\]
    If $\iota : \Pic(X) \hookrightarrow \bPic(X/F)(F)$ is the natural inclusion of Remark \ref{rmk_incpic}, we also define
    \[\Pic^0(X) := \iota^{-1}\left(\bPic^0(X/F)(F)\right).\]
\end{defn}
A line bundle $L$ lies in $\Pic^0(X)$ if and only if it is algebraically equivalent to $\cO_X$, see for example \cite[Proposition 9.5.10]{fantechi2005fundamental}. When the context is clear enough, we will write more simply $\bPic(X)$ (respectively, $\bPic^0(X)$) to refer to $\bPic(X/F)$ (respectively, $\bPic^0(X/F)$).

\begin{defn}
    Let $f : X \to S$ be a fibration between complex quasi-projective varieties. If $\eta \in S$ denotes the generic point of $S$, we denote by $\overline{\eta} \in S$ the image of the composition $\Spec \overline{\C(S)} \to \Spec \C(S) \to S$, where $\C(S)$ is the function field of $S$. We call $\overline{\eta}$ the \emph{geometric generic point} of $S$, and $X_{\overline{\eta}}$ the \emph{geometric generic fibre}.
\end{defn}
We now state the results on $K$-trivial fibrations that will be needed later. The following theorem was originally stated for $\Q$-divisors in \cite{xu2020homogeneous}, but Li observed in \cite[Remark 4.2]{li2023relative} that the statement holds more generally for $\R$-divisors.

\begin{thm}[Theorem 4.5 - \cite{xu2020homogeneous}]\label{thm_xu}
    Let $(X,\Delta)$ be a projective log pair with klt singularities over an algebraically closed field $F$ of characteristic zero. Assume that $K_X + \Delta \sim_{\R} 0$. Then $\textbf{Aut}^0(X,\Delta)$ is an abelian variety of dimension $h^1(X,\cO_X)$.
\end{thm}
The following lemma is \cite[Lemma 4.4]{li2023relative} reformulated differently with a first part that does not assume that the field over which the varieties are defined is algebraically closed. Recall that if $G$ and $H$ are group schemes, a morphism $G \to H$ is a \emph{homomorphism} if it induces a group homomorphism $G(S) \to H(S)$ for any scheme $S$, where $G(S)$ (respectively, $H(S)$) is the group of $S$-valued points of $G$ (respectively, of $H$).

\begin{lem}\label{lem_li44}
    Let $(X,\Delta)$ be a projective log pair over a field $F$ of characteristic $0$. Let $D$ be a Cartier divisor on $X$ and $G \subset \textbf{Aut}^0(X)$ a closed subgroup scheme. If $G$ is of finite type over $F$ and satisfies $H^0(G,\cO_G) = F$, then the following map
    \[\vartheta_D : G \to \bPic^0(X), ~[g] \mapsto [g^*D - D]\]
    exists and is a homomorphism of algebraic groups. If we assume furthermore that $F$ is algebraically closed, that $K_X + \Delta \sim_{\mathbb{R}} 0$, and that $D$ is big and nef, then the kernel of $\vartheta_D$ is finite.
\end{lem}
\begin{proof}
    Denote by $\alpha : G \times X \to X$ the natural group-scheme action of $G$ on $X$ and let $q : G \times X \to X$ be the projection. The line bundle $L := \alpha^*\cO_X(D) \otimes q^*\cO_X(-D)$ defines an element of $\bPic(X)(G)$ which corresponds, by representability of the Picard functor, to the morphism of $F$-schemes
    \[\tilde{\vartheta}_D : G \to \bPic(X), ~[g] \mapsto [g^*D - D].\]
    By \cite[Proposition 3.3.4]{brion2017structure}, this morphism factors through $\bPic^0(X)$ and the resulting morphism
    \[\vartheta : G \to \bPic^0(X)\]
    is a homomorphism of algebraic groups. If $F$ is algebraically closed, $K_X + \Delta \sim_{\R} 0$ and $D$ is big and nef, then the rest of the proof is identical to the second part of the proof of \cite[Lemma 4.4]{li2023relative}.
\end{proof}

\begin{cor}\label{cor_phiD}
   Let $(X,\Delta)$ be a projective log pair over a field $F$ of characteristic $0$ satisfying $K_X + \Delta \sim_\R 0$. Let $D$ be a Cartier divisor on $X$. Then the following map
   \[\vartheta_D : \textbf{Aut}^0(X,\Delta) \to \bPic^0(X), ~[g] \mapsto [g^*D - D]\]
   exists and is a homomorphism of algebraic groups.
\end{cor}
\begin{proof}
    As in the proof of Lemma \ref{lem_li44}, we always have a morphism of $F$-schemes
    \[\widetilde{\vartheta}_D : \textbf{Aut}^0(X,\Delta) \to \bPic(X).\]
    Since $\textbf{Aut}^0(X,\Delta)$ is connected by Proposition \ref{prop_concomp}, this morphism factors through $\bPic^0(X)$. Let 
    \[\vartheta_D : \textbf{Aut}^0(X,\Delta) \to \bPic^0(X)\]
    be the resulting morphism. To verify that it is a homomorphism of algebraic groups, it suffices to check it by passing to the algebraic closure $\overline{F}$ of $F$. By Proposition \ref{prop_concomp}, we have
    \[\bPic^0(X) \otimes_F \overline{F} \simeq \bPic^0(X_{\overline{F}}) ~\text{and}~ \textbf{Aut}^0(X,\Delta) \otimes_F \overline{F} \simeq \textbf{Aut}^0(X_{\overline{F}},\Delta_{\overline{F}}).\]
    Moreover, $\textbf{Aut}^0(X_{\overline{F}},\Delta_{\overline{F}})$ is an abelian variety by Theorem \ref{thm_xu} and we directly check that the base change $\vartheta_{D, \overline{F}}$ of $\vartheta_D$ to $\overline{F}$ is equal to the homomorphism of algebraic groups $\vartheta_{D_{\overline{F}}}$ from Lemma \ref{lem_li44}, and this concludes the proof.
\end{proof}

\begin{lem}[Lemma 4.8 - \cite{li2023relative}]\label{lem_li48}
    Let $\phi : X \dashrightarrow Y$ be a birational map between $\Q$-factorial projective varieties and $\Delta$ a divisor on $X$. If $\phi$ is an isomorphism in codimension one, then there exists an isomorphism of groups
    \[\Aut^0(X_\eta,\Delta_\eta) \to \Aut^0(Y_\eta,\Delta_{Y,\eta}), ~g \mapsto g_Y := \phi \circ g \circ \phi^{-1}.\]
\end{lem}

\begin{thm}[Theorem 4.9 - \cite{li2023relative}]\label{thm_li49}
    Let $f : (X,\Delta) \to S$ be a projective klt $K$-trivial fibration and $\pi : X \to Y$ a fibration over $S$. Let $H$ be a nef and big Cartier divisor on $Y$ over $S$ and set $B := \pi^* H$. Then, for any Cartier divisor $\xi$ on $Y_\eta$ such that $\xi \equiv 0$, there exists $h \in \Aut^0(X_\eta,\Delta_\eta)$ and a positive integer $\ell$ such that
    \[h \cdot B_\eta - B_\eta \sim \ell \pi_\eta^* \xi,\]
    where $\pi_\eta$ is the base change of $\pi$ to $\C(S)$.
\end{thm}

\section{From the generic to the relative weak movable cone conjecture}\label{sect_gentorel}
In this section, we ensure that \cite[Theorems 1.4 and 1.8]{li2023relative} remain valid if we weaken a hypothesis in the statements. More precisely, if $f : (X,\Delta) \to S$ is a $K$-trivial fibration, instead of asking for the existence of good minimal models for effective klt pairs of any projective variety of dimension $\dim(X/S)$, we only ask for the existence of good minimal models for effective klt perturbations of the very general fibre $(X_s, \Delta \vert_{X_s})$.

\begin{lem}[Lemma 3.7 - \cite{li2023relative}]\label{lem_lins}
    Let $f : X \to S$ be a fibration between complex quasi-projective varieties and $U \subset S$ an open subset. The following linear maps are well defined.
    \begin{enumerate}
        \item $N^1(X/S)_\R \to N^1(X_U/U)_\R, [D] \mapsto [D_U]$.
        \item $N^1(X/S)_\R \to N^1(X_\eta)_\R, [D] \mapsto [D_\eta]$.
        \item $N^1(X/S)_\R \to N^1(X_{\overline{\eta}})_\R, [D] \mapsto [D_{\overline{\eta}}]$.
        \item $N^1(X_\eta)_\R \to N^1(X_{\overline{\eta}})_\R, [D] \mapsto [D_{\overline{\eta}}]$.
    \end{enumerate}
    Moreover, for any sufficiently small open subset $U \subset S$, we have $N^1(X_U/U)_\R \simeq N^1(X_\eta)_\R$ and they both map injectively to $N^1(X_{\overline{\eta}})_\R$.
\end{lem}
One of the ingredients in the proof of Theorem \ref{thm_A} is the use of the geography of models of a klt $K$-trivial fibration $f : (X,\Delta) \to S$. The next result is an improved version of \cite[Theorem 2.7]{li2025relative}.

\begin{thm}[Theorem 3.3 - \cite{horing2024relative}]\label{thm_geopol}
    Let $f : (X,\Delta) \to S$ be a klt $K$-trivial fibration. Assume that for a very general closed point $s \in S$, good minimal models exist for effective klt perturbations of $(X_s, \Delta \vert_{X_s})$. Then for any rational polyhedral cone $\Pi \subset \Eff(X/S)$, there is a finite decomposition
    \[\Pi = \bigcup_{i=1}^m \Pi_i\]
    together with finitely many $\Q$-factorial birational contractions $\phi_i : X \dashrightarrow X_i$ over $S$ such that, for any $i \in \{1, \dotsc, m\}$, the following conditions are satisfied.
    \begin{enumerate}
        \item $\overline{\Pi_i}$ is a rational polyhedral cone,
        \item For any effective divisor $B$ such that $[B] \in \Pi_i$ and $(X,\Delta + \epsilon B)$ is klt for some $\epsilon > 0$, the birational map 
        \[\phi_i : (X,\Delta + \epsilon B) \dashrightarrow (X_i, \Delta_i + \epsilon B_i)\]
        is a good minimal model over $S$, where $\Delta_i := (\phi_i)_*\Delta$ and $B_i := (\phi_i)_* B$.
    \end{enumerate}
\end{thm}

\begin{rmk}
    The original statement of \cite[Theorem 3.3]{horing2024relative} is not phrased in this way. The authors assume the existence of good minimal models for effective klt pairs on the very general fiber, but the proof only uses the slightly weaker assumption that good minimal models exist for effective klt perturbations of $(X_s,\Delta \vert_{X_s})$.
\end{rmk}
We can then weaken the assumption of \cite[Lemma 2.5]{li2023relative} to obtain the following lemma.

\begin{lem}\label{lem_existsubcone}
    Let $f : (X, \Delta) \to S$ be a klt $K$-trivial fiber space. Assume that for a very general closed point $s \in S$, good minimal models exist for effective klt perturbations of $(X_s,\Delta \vert_{X_s})$. Let $P \subset \Eff(X/S)$ be a rational polyhedral cone, and let $\Gamma \subset \psaut(X/S,\Delta)$ be a subgroup. Then there exists a rational polyhedral subcone $Q \subset P \cap \Mov(X/S)$ such that
    \begin{enumerate}
        \item $\Gamma \cdot (Q \cap N^1(X/S)_\Q) = (\Gamma \cdot P) \cap \Mov(X/S)_\Q$, and
        \item $\Gamma \cdot Q = (\Gamma \cdot P) \cap \Mov(X/S)$.
    \end{enumerate}
\end{lem}
\begin{proof}
We argue as in \cite[Lemma 2.5]{li2023relative}.

(1) Let $P = \cup_{i = 1}^m P_i$ be a decomposition of $P$ in subcones as in Theorem \ref{thm_geopol}. The idea behind the construction of $Q$ is to retain only those $P_i$ that contain a rational point of $\Mov(X/S)$, and then take the convex cone generated by them. To make this precise, we first note the following fact:
\begin{equation}\label{eq_Qimov}
    \text{if} ~P_i \cap \Mov(X/S)_\Q \neq \emptyset, ~\text{then}~ P_i \cap N^1(X/S)_\Q \subset \Mov(X/S)_\Q.
\end{equation}
Indeed, take $[D] \in P_i \cap \Mov(X/S)_\Q$. Since $\Mov(X/S) \subset \Eff(X/S)$, we can assume that $D$ is effective by Lemma \ref{lem_effrel}. Let $0 < \epsilon \ll 1$ be such that $(X, \Delta + \epsilon D)$ is klt. For simplicity, denote by $h : X \dashrightarrow Y$ the good minimal model $\phi_i$ associated with $P_i$ as in Theorem \ref{thm_geopol}. We can assume that $h$ is an isomorphism in codimension one by \cite[Lemma 2.3]{li2025relative}. Let $B \geq 0$ such that $[B] \in P_i$. Replacing $\epsilon$ with a smaller positive real number, we can assume that $(X,\Delta + \epsilon B)$ is klt. In particular, $h$ is also a minimal model of $(X,\Delta + \epsilon B)$, hence $K_Y + \Delta_Y + \epsilon B_Y \sim_{\Q,S} \epsilon B_Y$ is semi-ample over $S$. It follows that $B$, as the strict transform of $B_Y$, is movable over $S$. Define $I := \enstq{i = 1, \dotsc, m}{P_i \cap \Mov(X/S)_\Q \neq \emptyset}$ and set 
\[Q := \operatorname{Cone}\left( \cup_{i \in I} \overline{P_i} \right) \subset P,\] 
where the last inclusion holds because $P$ is closed (being polyhedral). Since the $\overline{P_i}$ are rational polyhedral cones, so is $Q$, and $Q \subset \Mov(X/S)$ by (\ref{eq_Qimov}). 

To prove (1), let $[D] \in (\Gamma \cdot P) \cap \Mov(X/S)_\Q$. By definition, there exists $\gamma \in \Gamma$ such that $\gamma [D] \in P$, hence $\gamma [D] \in P_i$ for some $i$. In particulier, $[D] \in \gamma^{-1}Q$, which shows that 
\[(\Gamma \cdot P) \cap \Mov(X/S)_\Q \subset \Gamma \cdot (Q \cap N^1(X/S)_\Q).\]
The reverse inclusion follows from $Q \subset P$.

(2) The equality $\Gamma \cdot Q = (\Gamma \cdot P) \cap \Mov(X/S)$ is proved in exactly the same way.
\end{proof}
Thanks to this lemma, the assumptions of \cite[Theorem 1.3]{li2023relative} can also be weakened, yielding the following theorem.

\begin{thm}\label{thm_existrpc}
    Let $f : (X,\Delta) \to S$ be a terminal $K$-trivial fiber space. Assume that good minimal models exist for effective klt perturbations of the very general fibre $(X_s,\Delta \vert_{X_s})$. If there exists a polyhedral cone $P_\eta \subset \Eff(X_\eta)$ such that
    \[\Mov(X_\eta) \subset \psaut(X_\eta, \Delta_\eta) \cdot P_\eta,\]
    then there exists a rational polyhedral cone $Q \subset \Mov(X/S)$ such that
    \[\psaut(X,\Delta) \cdot (Q \cap N^1(X/S)_\Q) = \Mov(X/S)_\Q.\]
\end{thm}
\begin{proof}
We argue as in \cite[Theorem 1.3]{li2023relative}. After possibly enlarging $P_\eta$, we may assume that it is a rational polyhedral cone since each extremal ray of $P_\eta$ is generated by an $\R$-divisor class that is a linear combination of integral Cartier divisor classes. As $f$ is a terminal $K$-fiber space, we have $\psaut(X/S,\Delta) \simeq \psaut(X_\eta,\Delta_\eta)$ by \cite[Lemma 2.6]{li2023relative} and we will use the same notation to denote both a pseudo-automorphism on $X/S$ and its restriction to the generic fiber. Moreover, as $\Aut^0(X_\eta,\Delta_\eta)$ is a subgroup of $\psaut(X_\eta,\Delta_\eta)$, we denote by $\Aut^0(X/S,\Delta)$ the image of $\Aut^0(X_\eta,\Delta_\eta)$ under the isomorphism $\psaut(X_\eta,\Delta_\eta) \xrightarrow{\simeq} \psaut(X/S,\Delta)$. The construction of the weak fundamental domain is carried out in several steps. A divisor $D$ on $X$ is \emph{vertical} if $f\left(\Supp D\right) \neq S$, where $\Supp D$ denotes the support of the divisor $D$.

\vspace{0.1cm}

\emph{Step 1.} We show that, up to replacing $P_\eta$ by a smaller rational polyhedral cone satisfying the condition in the statement of the theorem, there exists a cone $P \subset \Mov(X/S)$ which maps onto $P_\eta$ under the map $N^1(X/S)_\R \to N^1(X_\eta)_\R$ of Lemma \ref{lem_lins}.(2).

\vspace{0.1cm}

\noindent Let $\xi$ be an effective $\Q$-Cartier divisor on $X_\eta$ such that $[\xi] \in P_\eta$. There exists a unique effective $\Q$-Cartier divisor $D$ on $X$ such that $D_\eta = \xi$ and $\Supp D$ has no vertical components. We show that if $\xi$ is movable, then $D$ is movable over $S$. Indeed, up to replacing $\xi$ by a multiple, we can assume that $\codim \bsloc \abs{\xi} \geq 2$. If $\xi \sim \xi'$ on $X_\eta$, then $\xi - \xi' = \operatorname{div}(\alpha)$ for some rational function $\alpha \in \C(X)$. Note that if $E$ is a prime vertical divisor, then there exists an effectif vertical divisor $E'$ such that $-E \sim_{\Q,S} E'$. It follows that there exists an effective vertical divisor $F$ such that $D \sim_{\Q,S} D' + F$. Since $\codim \bsloc \abs{\xi} \geq 2$, there exists $k \geq 1$ effective divisors $\xi^1, \xi^2, \dotsc, \xi^k$ such that $\xi^i \sim \xi$ for $i = 1, \dotsc, k$ and
\[\codim_{X_\eta} \left( \Supp \xi \cap \Supp \xi^1 \cap \dotsc \cap \Supp \xi^k \right) \geq 2.\]
The above discussion shows that we can find $k$ effective divisors $D^1, D^2, \dotsc, D^k$ on $X$ such that $D^i_\eta = \xi^i$ and $D \sim_{\Q,S} D^i$ for $i = 1, \dotsc, k$ (but it is possible that the supports of the $D^i$ have vertical components). Since $\Supp D$ does not have vertical components, it follows that
\[\codim_{X} \left( \Supp D \cap \Supp D^1 \cap \dotsc \cap \Supp D^k \right) \geq 2,\]
thus proving that $D$ is movable over $S$.

Recall that by Lemma \ref{lem_lins}.(4), there exists an injective map 
\[\mu : N^1(X_\eta)_\R \to N^1(X_{\overline{\eta}})_\R\]
that maps $\Mov(X_\eta)$ into $\Mov(X_{\overline{\eta}})$. Set 
\[P_{\overline{\eta}} := \mu(P_\eta) ~\text{and}~ \Gamma := \enstq{g_{\overline{\eta}}}{g \in \psaut(X_\eta, \Delta_\eta)} \subset \psaut(X_{\overline{\eta}}, \Delta_{\overline{\eta}}).\]
By Lemma \ref{lem_existsubcone}, we can find $k \geq 1$ effective divisors $\xi^1, \xi^2, \dotsc, \xi^k$ on $X_\eta$ such that
\begin{itemize}
    \item each $\xi^{i}_{\overline{\eta}}$ is movable,
    \item and if $\Pi_{\overline{\eta}} = \operatorname{Cone}\left([\xi^{i}_{\overline{\eta}}] ~\vert~ i = 1, \dotsc, k \right) \subset P_{\overline{\eta}}$, then
    \begin{equation}\label{eq_existrpc_1}
        \Gamma \cdot \Pi_{\overline{\eta}} = (\Gamma \cdot P_{\overline{\eta}}) \cap \Mov(X_{\overline{\eta}}).
    \end{equation}
\end{itemize}
By \cite[Lemma 5.1]{li2023relative}, each $\xi^i$ is movable. Let $D^i$ be the unique effective divisor on $X$ such that $D^{i}_{\eta} = \xi^{i}$ and $\Supp D^i$ does not have vertical components. By the above discussion, each $D^i$ is movable over $S$. Define the rational polyhedral cone
\[\Pi := \operatorname{Cone}\left( [D^i] ~\vert~ i=1, \dotsc, k \right) \subset \Mov(X/S)\]
and let $\Pi_\eta$ be the image of $\Pi$ under the natural map $N^1(X/S) \to N^1(X_\eta)$ of Lemma \ref{lem_lins}.(2). By construction, we have $\Pi_\eta \subset P_\eta \subset \Eff(X_\eta)$. We claim that
\begin{equation}\label{eq_existrpc_2}
    \Mov(X_\eta) = \psaut(X_\eta, \Delta_\eta) \cdot \Pi_\eta.
\end{equation}
Indeed, the inclusion "$\supset$" is clear since $\Pi_\eta \subset \Mov(X_\eta)$ by construction and $\psaut(X_\eta, \Delta_\eta)$ preserves the movable cone. For the reverse inclusion "$\subset$", since $\Mov(X_\eta) \subset \psaut(X_\eta, \Delta_\eta) \cdot P_\eta$ by assumption and $\Gamma \cdot \Pi_{\overline{\eta}} = (\Gamma \cdot P_{\overline{\eta}}) \cap \Mov(X_{\overline{\eta}})$ by (\ref{eq_existrpc_1}), we have
\[\mu(\Mov(X_\eta)) \subset \Gamma \cdot \Pi_{\overline{\eta}}.\]
As $\mu$ is injective, we finally obtain $\Mov(X_\eta) \subset \psaut(X_\eta, \Delta_\eta) \cdot \Pi_\eta$ and thus the equality (\ref{eq_existrpc_2}) is proved. Therefore, up to replacing $P_\eta$ by $\Pi_\eta$ and changing our notation by setting $P = \Pi$, we can assume in the following that the following rational polyhedral cone
\[P := \operatorname{Cone}\left( [D^i] ~\vert~ i=1, \dotsc, k \right) \subset \Mov(X/S)\]
maps onto $P_\eta$ under the map $N^1(X/S)_\R \to N^1(X_\eta)_\R$ of Lemma \ref{lem_lins}.(2).

\vspace{0.1cm}

\emph{Step 2.} Let $V \subset N^1(X/S)_\R$ be the vector subspace spanned by vertical divisors. Note that $V$ is defined over $\Q$. In this step, we show that if there exists a rational polyhedral cone $\Pi \subset \Eff(X/S)$ with the following property:
\[\forall [B] \in \Mov(X/S)_\Q, ~[B_\eta] \in P_\eta ~ \Longrightarrow ~ [B] \in \left( \Aut^0(X/S,\Delta) \cdot \Pi \right) + V,\] 
then there exists a rational polyhedral cone $Q \subset \Pi + V$ satisfying the statement of the theorem.

\vspace{0.1cm}

\noindent Suppose that such a cone $\Pi$ exists and let $[D] \in \Mov(X/S)_\Q$. By assumption, there exists $g \in \psaut(X_\eta,\Delta_\eta)$ and $[B_\eta] \in P_\eta$ such that $g[B_\eta] = [D_\eta]$. It follows that $[(g^{-1} \cdot D)_\eta] \in P_\eta$. By assumption on $\Pi$, we have $[g^{-1} \cdot D] \in \left(\Aut^0(X/S,\Delta) \cdot \Pi \right) + V$. As $\Aut^0(X/S,\Delta) \subset \psaut(X/S,\Delta)$ and $\psaut(X/S,\Delta) \cdot V = V$, we have $[D] \in \psaut(X/S,\Delta) \cdot (\Pi + V)$. This shows in particular that $\Mov(X/S)_\Q \subset \psaut(X/S,\Delta) \cdot (\Pi + V)$. By Lemma \ref{lem_existsubcone}, there is then a rational polyhedral cone $Q \subset \Pi + V$ such that
\[\psaut(X/S,\Delta) \cdot \left(Q \cap N^1(X/S)_\Q\right) = \Mov(X/S)_\Q,\]
and this completes this step. In the next two steps, we will construct a cone $\Pi$ as above.

\vspace{0.1cm}

\emph{Step 3.} We begin by proving something weaker. Let $B$ be an effective $\Q$-Cartier divisor on $X$ such that $[B] \in \Mov(X/S)_\Q$ and $[B_\eta] \in P_\eta$. In this step, we construct a rational polyhedral cone $\Pi_{B}^{h_i}$ which depends on $B$ and certain automorphisms $h_i \in \Aut^0(X_\eta,\Delta_\eta)$ such that for any effective divisor $D$ with $D_\eta \equiv B_\eta$, we have
\begin{equation}\label{eq_stp3goal}
    [D] \in \left( \Aut^0(X/S,\Delta) \cdot \Pi_{B}^{h_i} \right) + V.
\end{equation}
It is sufficient to work with a multiple of $B$, so we can assume that $B$ is a Cartier divisor.

\vspace{0.1cm}

\noindent Let $B$ and $D$ as above. For some $0 < \epsilon \ll 1$, $(X,\Delta + \epsilon B)$ has a good minimal model $\phi : X \dashrightarrow Y$ over $S$ by Theorem \ref{thm_hxli} and Remark \ref{rmk_allgood}. Moreover, by \cite[Lemma 2.2]{li2025relative}, we can assume that $\phi$ is an isomorphism in codimension one. Let $\tau : Y \to Z$ be the contraction over $S$ induced by the semiample divisor $B_Y := \phi_*B$. Up to replacing $B_Y$ by a multiple of it, we can further assume that $B_Y = \tau^*H$ for an ample divisor $H$ on $Z$ over $S$. Since $D_ \eta \equiv B_\eta$ by assumption, we can find an open subset $U \subset S$ such that $D_Y \vert_U \equiv_{U} B_Y \vert_U$ by Lemma \ref{lem_lins} and this shows that $D_Y \vert_U$ is nef over $U$. Once again, by Remark \ref{rmk_allgood}, we deduce that $D_Y \vert_U$ is semiample over $U$. In particular, there exists an ample divisor $H'$ on $Z_U$ over $U$ such that $D_Y \vert_U = \tau_U^* H'$. This implies
\begin{equation}\label{eq_stp3.1}
    D_{Y,\eta} - B_{Y,\eta} = \tau_\eta^*(H_\eta - H'_\eta),
\end{equation}
and $\tau_\eta^*H_\eta \equiv \tau_\eta^*H'_\eta$. Since $\tau$ is surjective and proper, the same holds for $\tau_\eta$ by base change. The pullback map $\tau_\eta^*$ is therefore injective and it follows that $H_\eta - H'_\eta \equiv 0$ on $Z_\eta$. By Theorem \ref{thm_li49}, there exists $g_Y \in \Aut^0(Y_\eta,\Delta_{Y,\eta})$ and an integer $k \geq 1$ such that
\[g_Y \cdot B_{Y,\eta} - B_{Y,\eta} \sim k(D_{Y,\eta} - B_{Y,\eta}).\]
As $X_\eta \dashrightarrow X_\eta$ is an isomorphism in codimension one, Lemma \ref{lem_li48} implies that
\[\Aut^0(X_\eta, \Delta_\eta) \xrightarrow{\simeq} \Aut^0(Y_\eta, \Delta_{Y,\eta}), ~g \mapsto g_Y := \phi \circ g \circ \phi^{-1}\]
is an isomorphism. Therefore $g \cdot B_\eta - B_\eta \sim k(D_\eta - B_\eta)$ and it follows that
\begin{equation}\label{eq_stp3.2}
    g \cdot B - B \sim k(D-B) ~\text{mod (vertical divisors)}.
\end{equation}
We claim that under the following map induced by Lemma \ref{lem_lins}
\[\pi : \Eff(X/S)/V \hookrightarrow N^1(X/S)_\R / V \to N^1(X_\eta)_\R,\]
there exists a finite-dimensional vector subspace $N$ such that
\begin{equation}\label{eq_stp3.3}
    \pi^{-1}\left([B_\eta]\right) = [B] + N.
\end{equation}
Indeed, up to replacing $X_\eta$ and $B_\eta$ by $Y_\eta$ and $B_{Y,\eta}$, we can assume that $B_\eta$ is semiample and we denote by $h : X_\eta \to Z_\eta$ the associated contraction. Define
\begin{equation}
    N := \operatorname{Span}_\R \enstq{[h^*\xi]}{\xi \equiv 0 ~\text{on}~ Z_\eta}/V
\end{equation}
where, by abuse of notation, $[h^*\xi]$ denotes the divisor class $[\Xi] \in N^1(X/S)_\R / V$ such that $\Xi$ is a divisor on $X$ satisfying $h^*\xi \sim_\Q \Xi_\eta$. Moreover, the class $[\Xi]$ depends uniquely on $\xi$.

We prove the equality (\ref{eq_stp3.3}). For the inclusion "$\subset$", take an effective $\Q$-Cartier divisor $D$ on $X$ such that $\pi([D]) = [B_\eta]$, that is, such that $D_\eta \equiv B_\eta$ on $X_\eta$. Arguing as above, we find as in (\ref{eq_stp3.2})
\[D \sim_{\Q} B + \dfrac{1}{k}(g \cdot B - B) ~\text{mod (vertical divisors)}\]
such that, by (\ref{eq_stp3.1}), we have
\[\dfrac{1}{k}(g \cdot B - B)_\eta \sim_{\Q} D_\eta - B_\eta.\]
Since $D_\eta \equiv B_\eta$, this proves that $[D] \in [B] + N$. For the reverse inclusion "$\supset$", any $\Q$-Cartier divisor $D$ on $X$ with $[D] \in [B] + N$ is $\Q$-linearly equivalent to an effective divisor since $B_\eta$ is, up to a multiple, the pullback of ample divisor on $Z_\eta$. This proves the desired equality. Furthermore, $N$ is finite-dimensional since $N^1(X/S)_\R$ itself is finite-dimensional by \cite[IV.$\S$4 - Proposition 3]{kleiman1966toward}. 
Let $t \geq 1$ be an integer and $D_1, D_2, \dotsc, D_t$ be Cartier divisors on $X$ such that $D_{i,\eta} \equiv B_\eta$ and the classes $[D_i] - [B]$ span $N$. For $i = 1, \dotsc, t$, by (\ref{eq_stp3.2}), there exists an integer $k_i \geq 1$ and $h_i \in \Aut^0(X_\eta,\Delta_\eta)$ such that
\begin{equation}\label{eq_pav}
h_i \cdot B - B \sim k_i(D_i - B) ~\text{mod (vertical divisors)}.
\end{equation}
We define the following rational polytope
\begin{equation}\label{eq_stp3.4}
    Q_{B}^{h_i} := [B] + \sum_{i=1}^t [0,1] \cdot [h_i \cdot B - B] \subset N^1(X/S)_\R
\end{equation}
and the following rational polyhedral cone
\[\Pi_{B}^{h_i} := \operatorname{Cone}(Q_{B}^{h_i}).\]
Let $\widetilde{Q}_{B}^{h_i}$ and $\widetilde{\Pi}_{B}^{h_i}$ be the images of $Q_{B}^{h_i}$ and $\Pi_{B}^{h_i}$ respectively under the quotient map $N^1(X/S)_\R \to N^1(X/S)_\R /V$. We claim that
\begin{equation}\label{eq_stp3.5}
    \pi^{-1}\left([B]\right) \subset \langle h_i ~\vert~ i=1, \dotsc, t \rangle \cdot \widetilde{\Pi}_{B}^{h_i},
\end{equation}
where $\langle h_i ~\vert~ i=1, \dotsc, t \rangle$ is the subgroup of $\Aut^0(X_\eta, \Delta_\eta)$ generated by the $h_i$. Assuming the inclusion above, we then see that $\widetilde{\Pi}_{B}^{h_i}$ satisfies (\ref{eq_stp3goal}) and we are done. We now prove (\ref{eq_stp3.5}). By Corollary \ref{cor_phiD}, there exists a homomorphism of algebraic groups
\[\vartheta_{B_{Y,\eta}} : \textbf{Aut}^0(Y_\eta, \Delta_{Y,\eta}) \to \bPic^0(Y_\eta).\]
The corresponding map on the $\overline{\C(S)}$-points is the following morphism of abelian groups
\[\Aut^0(Y_\eta, \Delta_{Y,\eta}) \to \Pic^0(Y_\eta), ~h_Y \mapsto h_Y^*B_{Y,\eta} - B_{Y,\eta}.\]
By Lemma \ref{lem_li48}, we have an isomorphism of groups $\Aut^0(X_\eta,\Delta_\eta) \simeq \Aut^0(Y_\eta, \Delta_{Y,\eta})$, consequently
\[\Aut^0(X_\eta, \Delta_{\eta}) \to \Pic^0(X_\eta), ~h \mapsto h^*B_{\eta} - B_{\eta}\]
is still a morphism of abelian groups. Finally, since $h \cdot B_\eta = (h^{-1})^*B_\eta$ for any $h \in \Aut(X_\eta)$, we see that
\begin{equation}\label{eq_stp3.6}
    \Aut^0(X_\eta, \Delta_{\eta}) \to \Pic(X_\eta), ~h \mapsto h \cdot B_{\eta} - B_{\eta}
\end{equation}
is a morphism of abelian groups. If $\enstq{[D_i] - [B]}{i = 1, \dotsc, s}$ is a basis of $N$ for $s \leq t$, then $\langle h_i ~\vert~ i = 1, \dotsc, s \rangle$ acts on $[B] + N$ by translation according to (\ref{eq_stp3.6}). By definition of $Q_{B}^{h_i}$, we see that the polytope $\widetilde{Q}_{B}^{h_i}$ tiles $[B] + N = \pi^{-1}([B_\eta])$ under the action of $\langle h_i ~\vert~ i = 1, \dotsc, s \rangle$. This proves (\ref{eq_stp3.5}).

\vspace{0.1cm}

\emph{Step 4.} We globalize the construction of the cone $\Pi_{B}^{h_i}$ to obtain the cone $\Pi$ mentioned in Step 2.

\vspace{0.1cm}

\noindent For all $g \in \Aut^0(X/S,\Delta)$ and for all Cartier divisor $D$ on $X$, we define
\begin{equation}\label{eq_action}
    g \star D := [g \cdot D - D] \in N^1(X/S)/\left(V\cap N^1(X/S)\right).
\end{equation}
Recall that in Step 1, we have lifted $P_\eta$ to a rational polyhedral cone
\begin{equation}\label{eq_rationalcone}
    P = \operatorname{Cone}\left( [D^i] ~\vert~ i=1, \dotsc, k \right) \subset \Mov(X/S),
\end{equation}
which means that $P$ maps onto $P_\eta$ under the map $N^1(X/S)_\R \to N^1(X_\eta)_\R$ of Lemma \ref{lem_lins}.(2). We use the above action (\ref{eq_action}) to define the following map
\[G_g : \bigoplus_{i=1}^k \Q \cdot D^i \to N^1(X/S)_\Q/\left(V \cap N^1(X/S)_\Q\right), \sum_{i=1}^k a_i D^i \mapsto \sum_{i=1}^k a_i (g \star D^i).\]
Since $\Aut^0(X_\eta,\Delta_\eta) \simeq \Aut^0(X/S,\Delta)$ is an abelian group and
\begin{equation}\label{eq_stp4.1}
    \Aut^0(X_\eta, \Delta_{\eta}) \to \Pic(X_\eta), ~h \mapsto h \cdot D_{\eta} - D_{\eta}
\end{equation}
is a morphism of abelian groups for all Cartier divisor $D$ on $X$ by (\ref{eq_stp3.6}), we have $G_{g+h} = G_g + G_h$. Set $\Aut^0(X/S,\Delta)_\Q := \Aut^0(X/S,\Delta) \otimes_\Z \Q$ and, for $\tau = \sum_j r_j g_j \in \Aut^0(X/S,\Delta)_\Q$ and $D \in \oplus_{i=1}^k \Q \cdot D^i$, we define 
\[G_\tau(D) := \sum_j r_j(g_j \star D) \in N^1(X/S)_\Q / \left(V \cap N^1(X/S)_\Q \right).\]
It is not clear that this definition is independent of the decomposition $\tau = \sum_j r_j g_j$, so we need to prove that everything is well defined. Take another decomposition $\tau = \sum_t r'_t g'_t$ and choose an integer $m \geq 1$ such that both $mr_j$ and $mr'_t$ lie in $\Z$ for all indices $j$ and $t$; in particular, we have $m\tau \in \Aut^0(X/S,\Delta)$. Write $D = \sum_{i=1}^k a_i D^i$ with $a_1, \dotsc, a_k \in \Q$. Since (\ref{eq_stp4.1}) is a morphism of abelian groups, it follows that
\[(mr_jg_j) \star D^i = mr_j(g_j \star D^i), ~(mr'_tg'_t) \star D^i = mr'_t(g'_t \star D^i)\]
and
\[G_{m\tau}(D) = \sum_{i=1}^k a_i \left((m\tau) \star D^i\right) = \sum_{i=1}^k a_i \left( \sum_j mr_j(g_j \star D^i) \right) = \sum_{i=1}^k a_i \left( \sum_t mr'_t(g'_t \star D^i) \right).\]
This shows that the definition of $G_\tau(D)$ does not depend on the expression of $\tau$. Therefore, we obtain a well-defined map
\[\Aut^0(X/S,\Delta)_\Q \times \bigoplus_{i=1}^k \Q \cdot D^i \to N^1(X/S)_\Q / \left(V \cap N^1(X/S)_\Q \right), ~(\tau, D) \mapsto G_\tau(D)\]
inducing the following linear map
\[j : \Aut^0(X/S,\Delta)_\Q \to \Hom_\Q\left(\bigoplus_{i=1}^k \Q \cdot D^i, ~N^1(X/S)_\Q / \left(V \cap N^1(X/S)_\Q \right) \right).\]
Note that the $\Q$-vector space on the right is finite-dimensional. Choose $g_1, \dotsc, g_l \in \Aut^0(X/S,\Delta)$ such that $\{G_{g_1}, \dotsc, G_{g_l}\}$ is a basis of $j(\Aut^0(X/S,\Delta)_\Q)$. For each $D^i$, we define as in (\ref{eq_stp3.4}) a rational polytope $Q_{D^i}^{g_j}$ associated to $g_1, \dotsc, g_l$. Then, we define the following rational polytope
\begin{equation}\label{eq_stp4lambda}
    \Lambda := \sum_{i = 1}^k [0,1]Q_{D^i}^{g_j} \subset N^1(X/S)
\end{equation}
and, finally, we define the following rational polyhedral cone
\begin{equation}
    \Pi := \operatorname{Cone}(\Lambda).
\end{equation}
We claim that $\Pi$ satisfies the condition of Step 2, that is, we claim that for any $[D] \in \Mov(X/S)_{\Q}$, if $[D_\eta] \in P_\eta$, then $[D] \in \left(\Aut^0(X/S,\Delta) \cdot \Pi\right) + V$. First, if $B = \sum_i r_i D^i$ with $r_i \in \Q_{\geq 0}$, we know by Step 3 that there exists an integer $t \geq 1$ and $\enstq{h_i}{i = 1, \dotsc, t} \subset \Aut^0(X/S,\Delta)$ such that $\enstq{[h_i \cdot B - B]}{i = 1, \dotsc, t}$ generates $N$. By the choice of the $g_i$, we have
\[\operatorname{Span}_\Q \enstq{G_{h_i}}{i = 1, \dotsc, t} \subset \operatorname{Span}_\Q \enstq{G_{g_j}}{j = 1, \dotsc, l}\]
and, consequently, we also have
\[\operatorname{Span}_\Q \enstq{G_{h_i}(B)}{i = 1, \dotsc, t} \subset \operatorname{Span}_\Q \enstq{G_{g_j}(B)}{j = 1, \dotsc, l}.\]
The above inclusion shows in particular that $\enstq{[g_j \cdot B - B]}{j = 1, \dotsc, l}$ generates $N$. Therefore, the work done in Step 3 shows that
\begin{equation}\label{eq_stp4.2}
    \pi^{-1}\left([B_\eta]\right) \subset \Aut^0(X/S,\Delta) \cdot \widetilde{\Pi}^{g_i}_{B}.
\end{equation}
Then, note that if $[D] \in \Mov(X/S)_\Q$ is a rational class satisfying $[D_\eta] \in P_\eta$, then we can find $B = \sum_i r_i D^i$ with $r_i \in \Q_{\geq 0}$ such that $D_\eta \equiv B_\eta$ or, equivalently, such that $[D] \in \pi^{-1}\left([B_\eta]\right)$. Indeed, recall that we have lifted $P_\eta$ to a rational polyhedral cone $P \subset \Mov(X/S)$ generated by the $[D^i]$, see (\ref{eq_rationalcone}). Therefore, by (\ref{eq_stp4.2}), it is sufficient to show that $\Pi_{B}^{g_i} \subset \Pi$ for any $B = \sum_i r_i D^i$ with $r_i \in \Q_{\geq 0}$.

Finally, let $B$ be as above. For every $j = 1, \dotsc, l$ and for every real number $\mu_j \in [0,1]$, we have
\[B + \sum_{j=1}^l \mu_j(g_j \cdot B - B) = \sum_{i} r_i \left(D^i + \sum_{j=1}^l \mu_j(g_j \cdot D^i - D^i) \right).\]
By the definition (\ref{eq_stp3.4}) of $Q_{B}^{g_i}$ and by the definition (\ref{eq_stp4lambda}) of $\Lambda$, we have $\Pi^{g_i}_{B} \subset \Pi$ and we are done. As stated in Step 2, we now only need to apply Lemma \ref{lem_existsubcone} to obtain the cone $Q$ in the statement of the theorem, and this completes the proof.
\end{proof}
These versions with more flexible assumptions allow us to prove the following theorem which improves on \cite[Theorem 1.8]{li2023relative}.

\begin{thm}\label{thm_limovconj}
    Let $f : (X,\Delta) \to S$ be a terminal $K$-trivial fiber space. Assume that good minimal models exist for effective klt perturbations of the very general fibre $(X_s,\Delta \vert_{X_s})$. Denote by $\Gamma$ the image of $\psaut(X/S,\Delta)$ under the natural group homomorphism $\psaut(X/S,\Delta) \to \GL(N^1(X/S)_\R)$. If there exists a polyhedral cone $P_\eta \subset \Eff(X_\eta)$ such that
    \[\Mov(X_\eta) \subset \psaut(X_\eta,\Delta_\eta) \cdot P_\eta,\]
    then $\Mov^+(X/S)$ admits a weak rational polyhedral fundamental domain under the action of $\Gamma$. Moreover, if $\Mov(X/S)$ is non-degenerate, then $\Mov^+(X/S)$ admits a rational polyhedral fundamental domain under the action of $\Gamma$.
\end{thm}
\begin{proof}
     Let $W$ be the maximal vector space contained in $\clMov(X/S)$. By \cite[Proposition 3.8]{li2023relative}, $W$ is defined over $\Q$. If $p : N^1(X/S)_\R \to N^1(X/S)_\R / W$ is the quotient map, for every subset $A \subset N^1(X/S)_\R$, we denote by $\tilde{A}$ the image $p(A)$ of $A$ by $p$. By Theorem \ref{thm_existrpc}, there exists a rational polyhedral cone $Q \subset \Mov(X/S)$ such that
    \begin{equation}\label{eq_egalcone}
        \psaut(X,\Delta) \cdot (Q \cap N^1(X/S)_\Q) = \Mov(X/S)_\Q.
    \end{equation}
    Note that $\widetilde{\Mov(X/S)}$ is non-degenerate and $\widetilde{Q} \subset \widetilde{\Mov^+(X/S)}$. The above equality (\ref{eq_egalcone}) implies that
    \[\widetilde{\Mov(X/S)}_\Q \subset \psaut(X,\Delta) \cdot \widetilde{Q},\]
    therefore $(\widetilde{\Mov^+(X/S)}, \widetilde{\Gamma_B})$ is of polyhedral type by Proposition \ref{prop_poltype}.(3), where $\widetilde{\Gamma_B}$ is the image of $\psaut(X/S,\Delta)$ in $\GL(N^1(X/S)_\R/W)$. It follows that there exists a weak rational polyhedral fundamental domain $\Pi \subset \Mov^+(X/S)$ for $\Mov^+(X/S)$ under the action of $\Gamma_B$ by Proposition \ref{prop_li3.6} such that
    \[\enstq{\gamma \in \Gamma_B}{\gamma \Pi = \Pi} = \enstq{\gamma \in \Gamma_B}{\gamma \text{ acts trivially on } N^1(X/S)_\R / W}.\]
    In particular, if $W=0$, then $\Pi$ is a rational polyhedral fundamental domain.
\end{proof}
Similarly, we obtain the following version of \cite[Theorem 1.5]{li2023relative}.

\begin{thm}\label{thm_fincontmov}
    Let $f : (X, \Delta) \to S$ be a klt $K$-trivial fiber space. Assume that good minimal models exist for effective klt perturbations of the very general fibre $(X_s,\Delta \vert_{X_s})$. If there exists a rational polyhedral cone $Q \subset \Mov(X/S)$ such that
    \[\Mov(X/S)_\Q \subset \psaut(X/S,\Delta) \cdot Q,\]
    then the set $\enstq{Z}{X \to Z ~\text{is a contraction over}~ S}$ is finite.
\end{thm}
\begin{proof}
    We argue as in \cite[Theorem 1.5]{li2023relative}. Define
    \[C := \operatorname{Cone}\enstq{g^*H}{g : X \to Z ~\text{over}~ S, ~H ~\text{an ample Cartier divisor over}~S}\]
    and note that $C \subset \Mov(X/S)$. By assumption, we have
    \begin{equation}\label{eq_cpsaut}
        C \cap N^1(X/S)_\Q \subset \psaut(X/S,\Delta) \cdot Q.
    \end{equation}
    Let $Q = \cup_{i = 1}^m Q_i$ be a decomposition of $Q$ in subcones as in Theorem \ref{thm_geopol} and denote by $\phi_i : X \dashrightarrow X_i$ the associated $\Q$-factorial birational contractions over $S$ for $i = 1, \dotsc, m$. By \cite[Lemma 2.3]{li2025relative}, we may assume that each $\phi_i$ is an isomorphism in codimension one. For simplicity, denote by $Q_i^{X_i}$ the strict transform by $\phi_i$ of the rational polyhedral cone $\overline{Q_i}$. Observe that $Q_i^{X_i} \subset \Nef(X_i/S)$. It follows that there are only finitely many contractions 
    \[h_i^{X_{i_k}} : X_i \to Y_i^{i_k}\]
    over $S$ corresponding to the faces of $Q_i^{X_i}$. We claim that if $g : X \to Y$ is a contraction over $S$, then $Y$ is isomorphic over $S$ to one of the $Y_i^{i_k}$.
    
    Indeed, let $H$ be an ample Cartier divisor on $Z$. By (\ref{eq_cpsaut}), there exists $\tau \in \psaut(X/S,\Delta)$ such that $\tau_*g^*H \in Q$. Let $i \in \{1, 2, \dotsc, m\}$ be an index such that $\tau_*g^*H \in Q_i$. By Theorem \ref{thm_geopol}.(2), the $\Q$-divisor $\phi_{i,*}\tau_*g^*H$ is semiample over $S$. After replacing it by a suitable multiple, we may assume that it is Cartier. Let $i_k$ be an index such that
    \[Y_i^{i_k} = \operatorname{Proj}_S R(X_i, \phi_{i,*}\tau_*g^*H), ~\text{where}~ R(X_i, \phi_{i,*}\tau_*g^*H) := \bigoplus_{n \geq 0} H^0(X_i, n\phi_{i,*}\tau_*g^*H).\]
    Similarly, we have $Y = \operatorname{Proj}_S R(X,g^*H)$. Since $\phi_i \circ \tau : X \dashrightarrow X \dashrightarrow X_i$ is a composition of birational maps that are isomorphisms in codimension one, we obtain
    \[R(X,g^*H) \simeq R(X,\tau_*g^*H) \simeq R(X_i, \phi_{i,*}\tau_*g^*H).\]
    Therefore $Y$ is isomorphic to $Y_i^{i_k}$ over $S$ and this completes the proof.
\end{proof}
An immediate consequence of the theorem is then this version of \cite[Corollary 1.6]{li2023relative}.

\begin{cor}\label{cor_fincontgen}
   Let $f : (X, \Delta) \to S$ be a terminal $K$-trivial fiber space. Assume that good minimal models exist for effective klt perturbations of the very general fibre $(X_s,\Delta \vert_{X_s})$. If there exists a rational polyhedral cone $P_\eta \subset \Eff(X_\eta)$ such that
    \[\Mov(X_\eta) \subset \psaut(X_\eta,\Delta_\eta) \cdot P_\eta,\]
    then the set $\enstq{Z}{X \to Z ~\text{is a contraction over}~ S}$ is finite. 
\end{cor}
\begin{proof}
    By Theorem \ref{thm_existrpc}, there exists a rational polyhedral cone $Q \subset \Mov(X/S) \subset \Eff(X/S)$ such that $\Mov(X/S)_\Q \subset \psaut(X/S,\Delta) \cdot Q$. The conclusion follows from Theorem \ref{thm_fincontmov}.
\end{proof}
Finally, as in \cite[Theorem 1.4]{li2023relative}, we are able to prove the finiteness of minimal models up to isomorphism.

\begin{thm}\label{thm_finmm}
    Let $f : (X, \Delta) \to S$ be a terminal $K$-trivial fiber space. Assume that good minimal models exist for effective klt perturbations of the very general fibre $(X_s,\Delta \vert_{X_s})$. If there exists a rational polyhedral cone $P_\eta \subset \Eff(X_\eta)$ such that
    \[\Mov(X_\eta) \subset \psaut(X_\eta,\Delta_\eta) \cdot P_\eta,\]
    then $(X,\Delta)$ has only finitely many minimal models over $S$, up to isomorphism.
\end{thm}
\begin{proof}
    We argue as in \cite[Theorem 1.4]{li2023relative}. By Theorem \ref{thm_existrpc}, there exists a rational polyhedral cone $Q \subset \Mov(X/S)$ such that
    \begin{equation}\label{eq_conemovdom}
        \Mov(X/S)_\Q \subset \psaut(X,\Delta) \cdot Q.
    \end{equation}
    Let $Q = \cup_{i = 1}^m Q_i$ be a decomposition of $Q$ in subcones as in Theorem \ref{thm_geopol} and denote by $\phi_i : X \dashrightarrow X_i$ for $i = 1, \dotsc, m$ the associated $\Q$-factorial birational contractions over $S$. By \cite[Lemma 2.3]{li2025relative}, for each $i$, there exist a $\Q$-factorial variety $Y_i$ projective over $S$, a small $\Q$-factorial modification $g_i : X \dashrightarrow Y_i$ over $S$ and a morphism $\nu_i : Y_i \to X_i$ over $S$ such that diagram
    \begin{center}
        \begin{tikzcd}
                        & X \arrow[rd, "\phi_i", dashed] \arrow[ld, "g_i"', dashed] &     \\
Y_i \arrow[rr, "\nu_i"] &                                                           & X_i
\end{tikzcd}
    \end{center}
    commutes. We first show that, up to isomorphism over $S$, there are only finitely many small $\Q$-factorial modifications of $X$. 
    
    Let $\sigma : X \dashrightarrow Y$ be a small $\Q$-factorial modification. Let $H$ be a relatively ample divisor on $Y$ over $S$ and denote by $H_X$ its strict transform on $X$. Note that $H_X \in \Mov(X/S)_\Q$. By (\ref{eq_conemovdom}) and the decomposition $Q = \cup_{i = 1}^m Q_i$, there exist $i \in \{1, \dotsc, m\}$ and $\tau \in \psaut(X,\Delta)$ such that $\tau_*H_X \in Q_i$. By construction of $Q_i$, the divisor $H_{X_i} := \phi_{i,*}\tau_*H_X$ is nef over $S$ on $X_i$. Consider the small $\Q$-factorial modification $\psi : Y \dashrightarrow Y_i$ given by the composition
    \[Y \overset{\sigma^{-1}}{\dashrightarrow} X \overset{\tau}{\dashrightarrow} X \overset{g_i}{\dashrightarrow} X_i.\]
    Note that $H_{X_i}$ is also the strict transform of $H$ by $\psi$. By \cite[Theorem 1]{kawamata2008flops}, $\psi$ can be decomposed as a sequence of flops. Since $H$ is ample, no such flop can exist, hence $Y \simeq X_i$ over $S$. This proves that there are only finitely many small $\Q$-factorial modifications of $X$ over $S$ up to isomorphism.
    
    Next, we show that $(X,\Delta)$ has finitely many minimal models over $S$ up to isomorphism. Let $\mu : (X,\Delta) \dashrightarrow (Y, \Delta_Y)$ be such a minimal model. By \cite[Lemma 2.3]{li2025relative}, there exist a $\Q$-factorial variety $Y'$ projective over $S$, a small $\Q$-factorial modification $\mu' : X \dashrightarrow Y'$ over $S$ and a morphism $\nu : Y' \to Y$ over $S$ such that the following diagram
    \begin{center}
        \begin{tikzcd}
                        & X \arrow[rd, "\mu", dashed] \arrow[ld, "\mu'"', dashed] &     \\
        Y' \arrow[rr, "\nu"] &                                                           & Y
        \end{tikzcd}
    \end{center}
    commutes. By the previous step, there exists $i \in \{1, \dotsc, m\}$ such that $Y' \simeq X_i$ over $S$. We therefore identify $\mu'$ to $g_i$. Since $g_i$ is an isomorphism in codimension one, we see that
    \[\Mov(Y_{i,\eta}) \subset \psaut(Y_{i,\eta}, \Delta_{Y_{i,\eta}}) \cdot P_{\eta,_i}\]
    where $P_{\eta,i}$ denotes the strict transform of $P_{\eta}$ on $Y_{i,\eta}$. By Corollary \ref{cor_fincontgen}, the set
    \[\enstq{Z}{Y_i \to Z ~\text{is a contraction over}~ S}\]
    is finite. Since $\nu$ belongs to this finite set, this shows that up to isomorphism over $S$, there are only finitely many minimal models of $(X,\Delta)$ over $S$.
\end{proof}
We now have all the tools we need to prove Theorem \ref{thm_C}.

\begin{proof}[Proof of Theorem \ref{thm_C}]
    The result is an immediate consequence of Theorem \ref{thm_limovconj} and Theorem \ref{thm_finmm}.
\end{proof}

\section{Product of good minimal models}\label{sect_prod}
In this section, we show that if (good) minimal models exist for effective klt perturbations of each factor of a product of $\Q$-factorial klt Calabi-Yau pairs $(X_i,\Delta_i)$ for $i=1,2$, then the same holds for the product, provided that all but possibly one factor have vanishing irregularity. This result is motivated by a key step in the proof of Theorem \ref{thm_A}. In that setting, the very general fibre of a $K$-trivial fibration is a quotient of a product $W = A \times \prod_i Y_i \times \prod_j S_j$, where $W$ has a boundary divisor $\Delta = \sum_j p_j^*\Delta_j$ such that the pairs $(S_j,\Delta_j)$ are klt Calabi-Yau pairs. Note that in this decomposition, the only factor with non-vanishing irregularity is the abelian variety $A$. To prove Theorem \ref{thm_A}, we will need to show that good minimal models exist for effective klt perturbations of $(W,\Delta)$. As we will see, this property holds for each factor of $W$, and the lemmas below explain how to deduce it for the product. The argument proceeds by induction on the number of factors of $W$.

\begin{lem}\label{lem_picsplit}
    Let $X_1$ and $X_2$ be projective complex varieties such that $h^1(X_2,\cO_{X_2}) = 0$. Denote by $p_1$ and $p_2$ the projections. The following properties hold.
    \begin{enumerate}
        \item The natural morphism 
        \[\Pic(X_1) \oplus \Pic(X_2) \to \Pic(X_1 \times X_2), ~(L_1,L_2) \mapsto p_1^*L_1 \otimes p_2^*L_2\]
        is an isomorphism. 
        
        \item The induced morphism
        \[N^1(X_1)_\R \oplus N^1(X_2)_\R \to N^1(X_1 \times X_2)_\R, ~([D_1],[D_2]) \mapsto p_1^*[D_1] + p_2^*[D_2]\]
        is also an isomorphism. Under this identification, we have 
        \[\Nef(X_1 \times X_2) = \Nef(X_1) \oplus \Nef(X_2).\]
    \end{enumerate}
\end{lem}
\begin{proof}
    Item (1) follows from \cite[Equation (4.25), p.123]{colliot2021brauer}. The same argument applies over $\C$ without assuming smoothness. For item (2), surjectivity follows from (1). To prove injectivity, assume that $\alpha := p_1^*[D_1]+p_2^*[D_2] = 0$ in $N^1(X_1 \times X_2)_\R$. Fix two closed points $x_1 \in X_1$ and $x_2 \in X_2$. For every irreducible curve $C_1$ on $X_1$, the projection formula yields the equality
    \[\alpha \cdot \left( C_1 \times \{x_2\}\right) = [D_1] \cdot C_1 = 0,\]
    hence $[D_1]=0$ in $N^1(X_1)_\R$. Similarly, intersecting with $\{x_1\}\times C_2$ for irreducible curves $C_2 \subset X_2$ gives $[D_2]=0$. This proves injectivity. The equality $\Nef(X_1 \times X_2)=\Nef(X_1) \oplus \Nef(X_2)$ follows from the same argument.
\end{proof}

\begin{cor}\label{cor_morisplit}
With the notation and assumptions of Lemma~\ref{lem_picsplit}, the natural morphism
\[N_1(X_1 \times X_2)_\R \to N_1(X_1)_\R \oplus N_1(X_2)_\R, \gamma \mapsto (p_{1*}\gamma, p_{2*}\gamma)\]
is an isomorphism and, under this identification, we have 
\[\overline{NE}(X_1 \times X_2) = \overline{NE}(X_1) \oplus \overline{NE}(X_2).\]
\end{cor}
\begin{proof}
    It follows directly from Lemma \ref{lem_picsplit} by duality.
\end{proof}

\begin{lem}\label{lem_prodklt}
    Let $(X_1, \Delta_1)$ and $(X_2,\Delta_2)$ be $\Q$-factorial klt Calabi-Yau pairs. Let $X := X_1 \times X_2$ and denote by $p_1 : X \to X_1$ and $p_2 : X \to X_2$ the projections. Let $\Delta := p_1^*\Delta_1 + p_2^*\Delta_2$. Then $(X,\Delta)$ is a $\Q$-factorial klt Calabi-Yau pair.
\end{lem}
\begin{proof}
    Since $X_1$ and $X_2$ are $\Q$-factorial, $X$ is also $\Q$-factorial by \cite[Théorème 6.5]{boissiere2011produit}. Moreover, $(X,\Delta)$ is klt by \cite[Lemma 4.2]{wei2023hyperbolicity}. Finally, we have
    \[K_X + \Delta \sim p_1^*(K_{X_1}+\Delta_1) + p_2^*(K_{X_2}+\Delta_2),\]
    thus $(X,\Delta)$ is a Calabi-Yau pair. This completes the proof.
\end{proof}

\begin{lem}\label{lem_goodminprod}
    Let $(X_1, \Delta_1)$ and $(X_2,\Delta_2)$ be $\Q$-factorial klt Calabi-Yau pairs, where $\Delta_1$ and $\Delta_2$ are $\Q$-divisors. Assume that $h^1(X_2,\cO_{X_2}) = 0$. Let $X := X_1 \times X_2$, and denote by $p_1 : X \to X_1$ and $p_2 : X \to X_2$ the projections. Let $\Delta := p_1^*\Delta_1 + p_2^*\Delta_2$. If (good) minimal models exist for effective klt perturbations of $(X_1,\Delta_1)$ and $(X_2,\Delta_2)$, then the same holds for $(X,\Delta)$.
\end{lem}
\begin{proof}
    First, $(X,\Delta)$ is a $\Q$-factorial klt Calabi-Yau pair by Lemma \ref{lem_prodklt}. Let $D$ be an $\R$-divisor on $X$ such that $(X,\Delta + D)$ is klt and $\kappa(X,K_X + \Delta + D) \geq 0$. We reduce to the case where $D$ is effective. Since $K_{X} + \Delta \equiv 0$, we have
    \[K_X + \Delta \sim_{\Q} 0\]
    by \cite[V. Corollary 4.9]{Nakayama2004} (see also \cite[Theorem 0.1]{ambro2005moduli}). Thus $\kappa(X,K_X + \Delta + D) \geq 0$ implies that for some integer $m \geq 1$, the divisor $\lfloor mD \rfloor$ is linearly equivalent to an effective divisor $E_m$. Let $D_m := \frac{1}{m}(E_m + (mD - \lfloor mD \rfloor)) \geq 0$ so that $D \equiv D_m$. For every $0 < \epsilon \ll 1$, the pair $(X,\Delta + \epsilon D_m)$ is klt and
    \[K_X + \Delta + D \equiv \frac{1}{\epsilon}(K_X + \Delta + \epsilon D_m).\]
    Replacing $D$ by $D_m$, we can therefore assume that $D$ is effective by \cite[Lemma 3.6.8]{birkar2010existence}.
    
    We now claim that there exist effective $\R$-divisors $D_1$ on $X_1$ and $D_2$ on $X_2$
    such that
    \begin{equation}\label{eq_divsplit}
        D \sim_{\R} p_1^*D_1 + p_2^*D_2.
    \end{equation}
    Indeed, if $D = 0$, there is nothing to prove. Otherwise, write $D = \sum_i a_i D^i$ with $a_i \in \R_{\geq 0}$ and $D^i$ a prime divisor on $X$. Since $X$ is $\Q$-factorial, we can assume that the $D^i$ are Cartier. By Lemma \ref{lem_picsplit}, we have 
    \[\Pic(X) = p_1^*\Pic(X_1) \oplus p_2^*\Pic(X_2).\]
    For each $i$, there exist Cartier divisors $D^i_1$ on $X_1$ and $D^i_2$ on $X_2$ such that 
    \[D^i \sim p_1^*D^i_1 + p_2^*D^i_2.\]
    By the Künneth formula, we can assume that $D^i_1$ and $D^i_2$ are effective. Let $D_1 := \sum_i a_i D^i_1$ and $D_2 := \sum_i a_i D^i_2$. Then $D_1$ and $D_2$ are effective, and (\ref{eq_divsplit}) follows immediately. 
    
    Let 
    \[B := p_1^*D_1 + p_2^*D_2 \geq 0.\]
    Since the pairs $(X,\Delta)$, $(X_1, \Delta_1)$ and $(X_2,\Delta_2)$ are klt, we can choose $0 \leq \epsilon \ll 1$ such that the pairs 
    \[(X, \Delta + \epsilon D), ~(X,\Delta + \epsilon B), ~(X_1,\Delta_1 + \epsilon D_1) ~\text{and}~ (X_2,\Delta_2 + \epsilon D_2)\]
    are all klt. As $K_{X} + \Delta \sim_{\Q} 0$ (respectively, $K_{X_i} + \Delta_i \sim_{\Q} 0$ for $i = 1,2$), the inequality $\kappa(X,K_X + \Delta + D) \geq 0$ (respectively, $\kappa(X_i, K_{X_i} + \Delta_i + D_i) \geq 0$) remains valid if we replace $D$ by $\epsilon D$ (respectively, $D_i$ by $\epsilon D_i$). Moreover, we have
    \[K_X + \Delta + D \equiv \frac{1}{\epsilon}(K_X + \Delta + \epsilon D) ~\text{and}~ K_{X_i} + \Delta_i + D_i \equiv \frac{1}{\epsilon}(K_{X_i} + \Delta_i + \epsilon D_i),\]
    thus a (good) minimal model for $(X, K_X + \Delta + D)$ (respectively, for $(X_i, K_{X_i} + \Delta_i + D_i)$) exists if and only if a (good) minimal model exists for $(X,K_X + \Delta + \epsilon D)$ (respectively, for $(X_i, K_{X_i} + \Delta_i + D_i)$) by \cite[Lemma 3.6.8]{birkar2010existence}. After replacing $D$, $D_1$, and $D_2$ by $\epsilon D$, $\epsilon D_1$, and $\epsilon D_2$, we can therefore assume that $\epsilon = 1$. By assumption, there exist (good) minimal models 
    \[\phi_1 : (X_1, \Delta_1 + D_1) \dashrightarrow (X'_1, \Delta'_1 + D'_1) ~\text{and}~ \phi_2 : (X_2, \Delta_2 + D_2) \dashrightarrow (X'_2, \Delta'_2 + D'_2).\]
    Define
    \[X' := X'_1 \times X'_2 ~\text{and}~ \phi := \phi_1 \times \phi_2 : X \dashrightarrow X'.\]
    Denote by $p'_1 : X' \to X'_1$ and $p'_2 : X' \to X'_2$ the projections. Then 
    \[B' := \phi_{*}B = {p'_{1}}^{*}D'_{1} + {p'_{2}}^{*}D'_{2}\]
    satisfies
    \[K_{X'} + \Delta' + B' \sim {p'_{1}}^*(K_{X'_1} + \Delta'_{1} + D'_{1}) + {p'_{2}}^*(K_{X'_2} + \Delta'_{2} + D'_{2}),\]
    and $K_{X'} + \Delta' + B'$ is therefore nef (and semiample) on $X'$. 
    
    We claim that $\phi : (X, \Delta + B) \dashrightarrow (X', \Delta' + B')$ is a (good) minimal model. By the above discussion, it suffices to show that $a(F,X,\Delta + B) > a(F,X',\Delta' + B')$ for all $\phi$-exceptional divisors $F$. By \cite[Lemma 3.6.3]{birkar2010existence}, this is equivalent to $\phi$ being $(K_X + \Delta + B)$-negative. Since $\phi_1$ is $(K_{X_1} + \Delta_1 + D_1)$-negative by the same lemma, there exists a common resolution
    \begin{center}
        \begin{tikzcd}
                               & W_1 \arrow[ld, "q_1"'] \arrow[rd, "r_1"] &    \\
X_1 \arrow[rr, "\phi_1", dotted] &                                      & X'_1
    \end{tikzcd} ~~~~
    \end{center}
    and a divisor $E_1$ on $W_1$ such that 
    \[q_1^*(K_{X_1} + \Delta_1 + D_1) = r_1^*(K_{X'_1} + \Delta'_1 + D'_{1}) + E_1,\]
    where ${q_1}_*E_1 \geq 0$ and whose support contains the union of all $\phi_1$-exceptional divisors. Similarly, since $\phi_2$ is $(K_{X_2} + \Delta_2 + D_2)$-negative, there exists a common resolution
    \begin{center}
        \begin{tikzcd}
                               & W_2 \arrow[ld, "q_2"'] \arrow[rd, "r_2"] &    \\
X_2 \arrow[rr, "\phi_2", dotted] &                                      & X'_2
\end{tikzcd}
    \end{center}
    and a divisor $E_2$ on $W_2$ such that 
    \[q_2^*(K_{X_2} + \Delta_2 + D_2) = r_2^*(K_{X'_2} + \Delta'_2 + D'_{2}) + E_2,\]
    with ${q_2}_*E_2 \geq 0$ and whose support contains the union of all $\phi_2$-exceptional divisors. Let $W := W_1 \times W_2$. Then $W$ is a common resolution of $X$ and $X'$ and we obtain the following commutative diagram.
    \begin{center}
        \begin{tikzcd}
                             & W \arrow[ld, "q ~:=~ q_1 \times q_2"'] \arrow[rd, "r ~:=~ r_1 \times r_2"] &    \\
X \arrow[rr, "\phi", dotted] &                                                      & X'
\end{tikzcd}
    \end{center}
    Denote by $\pi_1 : W \to W_1$ and $\pi_2 : W \to W_2$ the projections. We have 
    \[q^*(K_X + \Delta + B) = r^*(K_{X'} + \Delta' + B') + E,\]
    where $E := \pi_1^*E_1 + \pi_2^*E_2$, and 
    \[q_* E = p_1^*q_{1*}E_1 + p_2^*q_{2*}E_2 \geq 0.\]
    It remains to show that the support of $q_*E$ contains every $\phi$-exceptional prime divisor. Let $F$ be such a divisor. Since $\codim_{X_1} (p_1(F)) \leq 1$ and $\codim_{X_2} (p_2(F)) \leq 1$, there are three possibilities:
    \begin{itemize}
        \item $p_1(F)$ is a divisor on $X_1$,
        \item $p_2(F)$ is a divisor on $X_2$,
        \item $p_1(F) = X_1$ and $p_2(F) = X_2$.
    \end{itemize}
    The third case cannot occur, otherwise $F$ would not be $\phi$-exceptional. We treat only the first case, as the second is analogous. If $G := p_1(F)$ is a divisor, then $F \subset p_1^{-1}(G) = G \times X_2$. As both are irreducible of the same dimension, we have $F = G \times X_2$. Since $F$ is $\phi$-exceptional, we deduce from $\phi(F) = \phi_1(G) \times X_2'$ that $G$ is $\phi_1$-exceptional. Hence $G$ is contained in the support of $q_{1*}E_1$, and therefore 
    \[F = G \times X_2 \subset \Supp\left(p_1^*(q_{1*}E_1)\right) \subset \Supp(q_*E).\]
    This proves that $\phi : (X,\Delta + B) \dashrightarrow (X',\Delta' + B')$ is a (good) minimal model. 
    
    Finally, we show that $\phi$ is also a (good) minimal model for $(X,\Delta + D)$. Let $D' := \phi_*D$. Since $D \sim_{\R} B$, we have $D' \sim_{\R} B'$ and thus $K_{X'} + \Delta' + D'$ is nef (and semiample). It remains to check that $a(F,X,\Delta + D) > a(F,X',\Delta' + D')$ for all $\phi$-exceptional divisors $F$. Let $n\geq 1$, let $f_1, \dotsc, f_n \in \C(X)^\times$ be rational functions and let $a_1, \dotsc, a_n \in \R$ be such that 
    \[D = B + \sum_{i=1}^n a_i \operatorname{div}_{X}(f_i).\]
    Identifying $\C(X')$ with $\C(X)$ through the isomorphism induced by $\phi$, we have $D' = B' + \sum_{i=1}^n a_i\operatorname{div}_{X'}(f_i)$ by \cite[Proposition 1.4]{fulton1988intersection}. The discrepancy inequalities for $(X,\Delta + D)$ and $(X',\Delta' + D')$ therefore follow from those for $(X,\Delta + B)$ and $(X',\Delta' + B')$. This proves that $\phi : (X,\Delta + D) \dashrightarrow (X',\Delta' + D')$ is a (good) minimal model, and this completes the proof.
\end{proof}
The following lemma shows that, under the above assumptions, any divisorial contraction or flip appearing in a $(K_X+\Delta+D)$-MMP on the product $X = X_1 \times X_2$ is induced by a corresponding step of the MMP on one of the two factors.

\begin{lem}\label{lem_mmpprod1}
    Let $(X_1, \Delta_1)$ and $(X_2,\Delta_2)$ be $\Q$-factorial klt Calabi-Yau pairs, where $\Delta_1$ and $\Delta_2$ are $\Q$-divisors. Assume that $h^1(X_2,\cO_{X_2}) = 0$. Let $X := X_1 \times X_2$ and denote by $p_1 : X \to X_1$ and $p_2 : X \to X_2$ the projections. Let $\Delta := p_1^*\Delta_1 + p_2^*\Delta_2$ and let $D$ be an effective $\R$-divisor on $X$ such that $(X,\Delta + D)$ is klt. Let
    \begin{equation}\label{eq_stepX}
        \begin{tikzcd}
{X} \arrow[rr, "f", dashed] \arrow[rd, "\eta"'] &   & {X'} \arrow[ld, "\eta^+"] \\
                                                    & S &                               
\end{tikzcd}
    \end{equation}
    be either a divisorial contraction (in which case $X' = S$, $f = \eta$ and $\eta^+ = \id$) or a flip of a $(K_X+ \Delta + D)$-negative extremal ray. Then exactly one of the following two cases occurs.
    
    \begin{itemize}
        \item There exist an effective $\R$-divisor $D_1$ on $X_1$ such that $(X_1,\Delta_1 + D_1)$ is klt, a divisorial contraction or flip
        \begin{equation}\label{eq_step1}
            \begin{tikzcd}
            X_1 \arrow[rr, dashed, "f_1"] \arrow[rd, "\eta_1"'] & & X'_1 \arrow[ld, "\eta_1^+"] \\
            & S_1 &
            \end{tikzcd}
        \end{equation}
        associated to a $(K_{X_1} + \Delta_1 + D_1)$-negative extremal ray, and isomorphisms $S \simeq S_1 \times X_2$, $X' \simeq X'_1 \times X_2$ such that the diagram (\ref{eq_stepX}) is obtained from (\ref{eq_step1}) by base change along $\id_{X_2}$.
        
        \item There exist an effective $\R$-divisor $D_2$ on $X_2$ such that $(X_2,\Delta_2 + D_2)$ is klt, a divisorial contraction or flip
        \begin{equation}\label{eq_step2}
            \begin{tikzcd}
            X_2 \arrow[rr, dashed, "f_2"] \arrow[rd, "\eta_2"'] & & X'_2 \arrow[ld, "\eta_2^+"] \\
            & S_2 &
            \end{tikzcd}
        \end{equation}
        associated to a $(K_{X_2} + \Delta_2 + D_2)$-negative extremal ray, and isomorphisms $S \simeq X_1 \times S_2$, $X' \simeq X_1 \times X'_2$ such that the diagram (\ref{eq_stepX}) is obtained from (\ref{eq_step2}) by base change along $\id_{X_1}$.
    \end{itemize}
\end{lem}
\begin{proof}
    As in the proof of Lemma \ref{lem_goodminprod}, there exist effective $\R$-divisors $D_1$ on $X_1$ and $D_2$ on $X_2$ such that
    \begin{equation}\label{eq_decomp}
        D \sim_{\R} p_1^*D_1 + p_2^*D_2.
    \end{equation}
    By Corollary \ref{cor_morisplit}, the natural map
    \[N_1(X)_\R \to N_1(X_1)_\R \oplus N_1(X_2)_\R, ~~\gamma \mapsto (p_{1*}\gamma, p_{2*}\gamma)\]
    is an isomorphism and, under this identification, we have $\overline{NE}(X) = \overline{NE}(X_1) \oplus \overline{NE}(X_2)$. In particular, if $\R_{\geq 0}(\gamma_1,\gamma_2)$ is an extremal ray of $\overline{NE}(X)$, then necessarily $\gamma_1 = 0$ or $\gamma_2 = 0$ since otherwise the decomposition
    \[(\gamma_1, \gamma_2) = (\gamma_1, 0) + (0,\gamma_2)\]
    would contradict extremality. Let $R \subset \overline{NE}(X)$ be the $(K_X + \Delta + D)$-negative extremal ray associated with (\ref{eq_stepX}). By the above discussion and the projection formula, $R$ is of the form
    \[R = R_1 \oplus 0 ~\text{or}~ R = 0 \oplus R_2\]
    where $R_1$ (respectively, $R_2$) is a $(K_{X_1} + \Delta_1 + D_1)$-negative extremal ray of $\overline{NE}(X_1)$ (respectively, a $(K_{X_2} + \Delta_2 + D_2)$-negative extremal ray of $\overline{NE}(X_2)$). By symmetry, we can assume that $R = R_1 \oplus 0$. By Lemma \ref{lem_prodklt}, the pair $(X,\Delta)$ is a $\Q$-factorial klt Calabi-Yau pair. Since $(X_1,\Delta_1)$ and $(X_2,\Delta_2)$ are klt, we can choose $0 < \epsilon \ll 1$ such that the pairs
    \[(X,\Delta + \epsilon D), ~(X_1, \Delta_1 + \epsilon D_1) ~\text{and}~ (X_2,\Delta_2 + \epsilon D_2)\]
    are all klt. In addition, $R$ remains $(K_X + \Delta + \epsilon D)$-negative, so after replacing $D$, $D_1$ and $D_2$ with $\epsilon D$, $\epsilon D_1$ and $\epsilon D_2$, we can assume that $\epsilon = 1$. Let $\eta_1 : X_1 \to S_1$ be the contraction associated with $R_1$. It follows that 
    \[\eta_1 \times \id_{X_2} : X \to S_1 \times X_2\]
    contracts exactly the curves whose numerical classes lie in $R = R_1 \oplus 0$. By the uniqueness of extremal contractions \cite[Theorem 4.1.5(ii)]{fujino2017foundations}, we obtain
    \begin{equation}\label{eq_identif}
        S \simeq S_1 \times X_2 ~\text{and}~ \eta \simeq \eta_1 \times \id_{X_2}.
    \end{equation}
    Moreover
    \[\operatorname{Ex}(\eta) \simeq \operatorname{Ex}(\eta_1) \times X_2,\]
    so $\eta$ is a divisorial contraction (respectively a flipping contraction) if and only if $\eta_1$ is. From now on, we identify $S$ with $S_1 \times X_2$ and $\eta$ with $\eta_1 \times \id_{X_2}$.
    
    If $\eta$ is divisorial, then we are done by taking $X'_1 := S_1$, $f_1 := \eta_1$ and $\eta_1^+ := \id_{S_1}$. Assume now that $\eta$ is a flipping contraction. We show that the diagram (\ref{eq_stepX}) is obtained, up to isomorphism, from (\ref{eq_step1}) by base change along $\id_{X_2}$. Since flips are unique, it suffices to verify that
    \begin{enumerate}
        \item[(i)] $\codim_{X'_1 \times X_2} \operatorname{Ex}(\eta_1^+ \times \id_{X_2}) \geq 2$,
        \item[(ii)] $K_{X'_1 \times X_2} + \Delta' + D'$ is $\R$-Cartier and $(\eta_1^+ \times \id_{X_2})$-ample, where $\Delta'$ and $D'$ are the birational transforms of $\Delta$ and $D$ on $X'_1 \times X_2$.
    \end{enumerate}
    Item (i) holds because $\operatorname{Ex}(\eta_1^+ \times \id_{X_2}) \simeq \operatorname{Ex}(\eta_1^+) \times X_2$ and $\eta_1^+$ is small. For (ii), since $X'_1$ is $\Q$-factorial by \cite[Proposition 3.5, Proposition 3.7]{fujino2017foundations}, the product $X'_1 \times X_2$ is also $\Q$-factorial by \cite[Théorème 6.5]{boissiere2011produit} and therefore $K_{X'_1 \times X_2} + \Delta' + D'$ is $\R$-Cartier. Since $K_{X_1} + \Delta_1 \equiv 0$, we also have $K_{X'_1} + \Delta'_1 \equiv 0$. Since $D \equiv_S p_1^*D_1$ by (\ref{eq_decomp}) and (\ref{eq_identif}), if $D_1'$ denotes the birational transform of $D_1$ on $X'_1$, it follows that
    \begin{equation}\label{eq_equivS}
        K_{X'_1 \times X_2} + \Delta' + D' \equiv_S {p'_1}^*(K_{X'_1} + \Delta'_1 + D'_1),
    \end{equation}
    where $p'_1 : X'_1 \times X_2 \to X'_1$ is the projection. Now $K_{X'_1} + \Delta'_1 + D'_1$ is $\eta_1^+$-ample, hence ${p'_1}^*(K_{X'_1} + \Delta'_1 + D'_1)$ is $(\eta_1^+ \times \id_{X_2})$-ample. We conclude that $K_{X'_1 \times X_2} + \Delta' + D'$ is $(\eta_1^+ \times \id_{X_2})$-ample by (\ref{eq_equivS}) and \cite[IV.4. Theorem 1]{kleiman1966toward}. This completes the proof.
\end{proof}
In particular, any $(K_X+\Delta+D)$-MMP on such a product is induced by log-MMPs on the two factors.

\begin{cor}\label{cor_productMMP}
Let $(X_1, \Delta_1)$ and $(X_2,\Delta_2)$ be $\Q$-factorial klt Calabi-Yau pairs, where $\Delta_1$ and $\Delta_2$ are $\Q$-divisors. Assume that $h^1(X_2,\cO_{X_2}) = 0$. Let $X := X_1 \times X_2$ and denote by $p_1 : X \to X_1$ and $p_2 : X \to X_2$ the projections. Let $\Delta := p_1^*\Delta_1 + p_2^*\Delta_2$ and let $D$ be an effective $\R$-divisor on $X$ such that $(X,\Delta + D)$ is klt. Let
\begin{equation}
    (X^0,\Delta^0 + D^0) = (X,\Delta + D) \overset{f^0}{\dashrightarrow} (X^1,\Delta^1 + D^1) \overset{f^1}{\dashrightarrow} (X^2,\Delta^2 + D^2) \overset{f^2}{\dashrightarrow} \cdots
\end{equation}
be a $(K_X + \Delta +D)$-MMP. Then there exist effective $\R$-divisors $D_1$ on $X_1$ and $D_2$ on $X_2$ such that $(X_1,\Delta_1 + D_1)$ and $(X_2,\Delta_2 + D_2)$ are klt, together with sequences of $\Q$-factorial projective klt pairs
\[(X_1^i,\Delta_1^i + D_1^i)_{i \geq 0}, ~(X_2^i,\Delta_2^i + D_2^i)_{i \geq 0}\]
and birational maps
\[f_1^i : (X_1^i,\Delta_1^i + D_1^i) \dashrightarrow (X_1^{i+1},\Delta_1^{i+1} + D_1^{i+1}), ~ f_2^i : (X_2^i,\Delta_2^i + D_2^i) \dashrightarrow (X_2^{i+1},\Delta_2^{i+1} + D_2^{i+1})\]
such that:

\begin{enumerate}
\item $(X_1^0,\Delta_1^0 + D_1^0)=(X_1,\Delta_1 + D_1)$ and $(X_2^0,\Delta_2^0 + D_2^0)=(X_2,\Delta_2 + D_2)$.

\item For every $i \geq 0$, exactly one of the maps $f_1^i$ and $f_2^i$ is the identity, and the other is a divisorial contraction or a flip associated with a negative extremal ray of the corresponding pair $(X_1^i,\Delta_1^i + D_1^i)$ or $(X_2^i, \Delta_2^i + D_2^i)$.

\item For every $i \geq 0$, there exists an isomorphism $\phi_i: X^i \xrightarrow{\sim} X_1^i \times X_2^i$ under which either $f^i \simeq f_1^i \times \id_{X_2^i}$ or $f^i \simeq \id_{X_1^i}\times f_2^i$.

\item For every $i \geq 0$, $K_{X_1^i} + \Delta_1^i \equiv 0$, $K_{X_2^i} + \Delta_2^i \equiv 0$ and $h^1(X_2^i,\cO_{X_2^i}) = 0$.
\end{enumerate}

Moreover, if we remove the identity maps from the sequences $(f_1^i)_{i \geq 0}$ and $(f_2^i)_{i \geq 0}$, the remaining maps form respectively a $(K_{X_1}+\Delta_1+D_1)$-MMP on $(X_1,\Delta_1+D_1)$ and a $(K_{X_2}+\Delta_2+D_2)$-MMP on $(X_2,\Delta_2+D_2)$ (possibly empty).
\end{cor}
\begin{proof}
    We construct the sequences $(X_1^i,\Delta_1^i + D_1^i)_{i \geq 0}$, $(X_2^i,\Delta_2^i + D_2^i)_{i \geq 0}$ and the maps $f_1^i$, $f_2^i$ by induction on $i$. Assume that for some $i \geq 0$ we have already constructed
    \[(X_1^j,\Delta_1^j + D_1^j), ~(X_2^j,\Delta_2^j + D_2^j), ~\phi_j ~\text{for}~ 0\leq j\leq i, ~\text{and}~ f_1^j, f_2^j ~\text{for}~ 0\leq j < i\]
    The isomorphism $\phi_i : X^i \xrightarrow{\sim} X_1^i \times X_2^i$ identifies the divisors $\Delta^i$ and $D^i$ with $\R$-divisors on $X_1^i \times X_2^i$ which we still denote by $\Delta^i$ and $D^i$ for simplicity. Apply Lemma \ref{lem_mmpprod1} to the klt pair $(X_1^i \times X_2^i, \Delta^i + D^i)$ and to the step $(X^i,\Delta^i + D^i) \dashrightarrow (X^{i+1},\Delta^{i+1} + D^{i+1})$. There are two cases, depending on whether the contracted extremal ray lies in the Mori cone of $X_1^i$ or in that of $X_2^i$. Since the arguments are identical, we treat only the first case. We obtain an effective $\R$-divisor $D_1^{i+1}$ on a normal projective $\Q$-factorial variety $X_1^{i+1}$ such that $(X_1^{i+1},\Delta_1^{i+1} + D_1^{i+1})$ is klt, together with a birational map
    \[f_1^i : (X_1^i,\Delta_1^i + D_1^i) \dashrightarrow (X_1^{i+1},\Delta_1^{i+1} + D_1^{i+1})\]
    which is either a divisorial contraction or a flip of a $(K_{X_1^i}+\Delta_1^i+D_1^i)$-negative extremal ray, and an isomorphism
    \[\phi_{i+1} : X^{i+1} \xrightarrow{\sim} X_1^{i+1} \times X_2^i\]
    under which $f^i$ corresponds to $f_1^i \times \id_{X_2^i}$. Set $X_2^{i+1} := X_2^i$, $D_2^{i+1} := D_2^i$, and define $f_2^i := \id_{X_2^i}$. This proves the first three items. For (4), by induction we have $K_{X_1^i} + \Delta_1^i \equiv 0$ and $K_{X_2^i} + \Delta_2^i \equiv 0$. For $\ell=1,2$, each pair $(X_\ell^i,\Delta_\ell^i + D_\ell^i)$ appearing in a $(K_{X_\ell}+\Delta_\ell+D_\ell)$-MMP still satisfies $K_{X_\ell^i}+\Delta_\ell^i \equiv 0$ by \cite[Theorem 3.7(4)]{Kollár_Mori_1998}. Since $(X_2^i,\Delta_2^i+D_2^i)$ and $(X_2^{i+1},\Delta_2^{i+1}+D_2^{i+1})$ are klt, they have rational singularities by \cite[Theorem 4.9]{fujino2017foundations}, hence $h^1(X_2^{i+1},\cO_{X_2^{i+1}}) = h^1(X_2^i,\cO_{X_2^i}) = 0$. This completes the proof.
\end{proof}
As an immediate consequence, termination of log-MMPs on the two factors implies termination of log-MMPs on the product.

\begin{cor}\label{cor_termprodgood}
Let $(X_1, \Delta_1)$ and $(X_2,\Delta_2)$ be $\Q$-factorial klt Calabi-Yau pairs, where $\Delta_1$ and $\Delta_2$ are $\Q$-divisors. Assume the following conditions.
    \begin{enumerate}
        \item $h^1(X_2,\cO_{X_2}) = 0$.
        \item For $i = 1,2$ and every $\R$-divisor $D_i$ on $X_i$ such that $(X_i,\Delta_i + D_i)$ is klt and $\kappa(X_i,K_{X_i} + \Delta_i + D_i) \geq 0$, any log-MMP for $(X_i,\Delta_i + D_i)$ terminates with a (good) minimal model.
    \end{enumerate}
Let $X := X_1 \times X_2$ and denote by $p_1 : X \to X_1$ and $p_2 : X \to X_2$ the projections. Let $\Delta := p_1^*\Delta_1 + p_2^*\Delta_2$ and let $D$ be an $\R$-divisor such that $(X,\Delta + D)$ is klt and $\kappa(X,K_X + \Delta + D) \geq 0$. Then any MMP for $(X,\Delta + D)$ terminates with a (good) minimal model.
\end{cor}
\begin{proof}
    Termination of any MMP for $(X,\Delta + D)$ follows directly from Corollary \ref{cor_productMMP}. If in addition, any MMP for effective klt perturbations of $(X_i,\Delta_i)$ terminates with a good minimal model for $i = 1,2$, then the same holds for $(X,\Delta)$ by Lemma \ref{lem_goodminprod} and Remark \ref{rmk_allgood}.
\end{proof}

\section{Proof of the main theorem}\label{sect_proof}
In this section, we prove Theorem \ref{thm_A} and Corollary \ref{cor_B}. We will need the following lemma.

\begin{lem}\label{lem_ggfiber}
    Let $f : X \to S$ be a fibration between normal quasi-projective varieties. Identifying $N^1(X_\eta)_\R$ to a vector subspace of $N^1(X_{\overline{\eta}})_\R$ through the natural injective morphism $N^1(X_\eta)_\R \to N^1(X_{\overline{\eta}})_\R$ of Lemma \ref{lem_lins}.(4), we have the following equalities.
    \begin{enumerate}
        \item $\Eff(X_\eta) = \Eff(X_{\overline{\eta}}) \cap N^1(X_\eta)_\R$.
        \item $\Mov(X_\eta) = \Mov(X_{\overline{\eta}}) \cap N^1(X_\eta)_\R$.
        \item $\clMov(X_\eta) = \clMov(X_{\overline{\eta}}) \cap N^1(X_\eta)_\R$.
    \end{enumerate}
\end{lem}
\begin{proof}
    We start with items (1) and (2). For the inclusion "$\subset$", it is sufficient to prove that the pullback of an effective divisor (respectively, of a movable divisor) on $X_{\eta}$ remains effective (respectively, movable) on $X_{\overline{\eta}}$. The case of movable divisors has been dealt with in \cite[Lemma 5.1]{li2023relative}, so we will focus on the case of effective divisors. Take an effective Cartier divisor $D$ on $X_\eta$. By the flat base change theorem \cite[III - Proposition 9.3]{hartshorne1977graduate}, we have
    \[H^0\left(X_{\overline{\eta}}, \cO_{X_{\overline{\eta}}}(D_{\overline{\eta}})\right) \simeq H^0\left(X_\eta, \cO_{X_\eta}(D)\right) \otimes_F \overline{F} \neq (0).\]
    Therefore, $D_{\overline{\eta}}$ is effective. For the reverse inclusion "$\supset$", we will only discuss the case of effective classes, since the case of movable classes can be proven in exactly the same way. Let $[D] = \sum a_i [D_i]$ be an element of $\Eff(X_{\overline{\eta}}) \cap N^1(X_\eta)_\R$, where the $a_i$ are positive real numbers and the $D_i$ are effective Cartier divisors on $X_{\overline{\eta}}$. Take a finite Galois extension $F'/F$ such that the $D_i$ comes from effective divisors on $X_{\eta, F'} := X_{\eta} \otimes_F F'$ and let $G := \Gal(F'/F)$. For each index $i$, we set 
    \[D_i' = \sum_{\sigma \in G} \sigma^*D_i.\]
    Note that the $D_i'$ are effective $\Gal(\overline{F}/F)$-invariant divisors on $X_{\overline{\eta}}$, therefore they come from effective divisors defined on $X_\eta$. In particular, we deduce that the $[D_i']$ lie in $\Eff(X_\eta)$. Since $[D]$ is $\Gal(\overline{F}/F)$-invariant, it follows that 
    \[[D] = \dfrac{1}{\abs{G}} \sum a_i [D_i'] \in \Eff(X_\eta),\]
    which proves the equality of items (1) and (2) in the statement of the lemma. Finally, the equality between the cones in item (3) can be deduced from the following lemma.
\end{proof}
\begin{lem}[Lemma 3.8 - \cite{bright2020finiteness}]\label{lem_coneboundary}
    Let $V$ be a real vector space, let $C \subset V$ be a convex cone, and let $S \subset V$ be a subspace having non-empty intersection with the interior of $C$. Then we have
    \[\partial_S(C \cap S) = \partial_V(C) \cap S.\]
\end{lem}
\begin{rmk}\label{rmk_invariantcommute}
    Lemma \ref{lem_coneboundary} is stated in \cite{bright2020finiteness} for \emph{closed} convex cones, but it remains true more generally for convex cones $C$ satisfying the assumption $S \cap C^\circ \neq \emptyset$.
\end{rmk}
To simplify the proof of Theorem \ref{thm_A}, we will need the following lemma, which is certainly well known.

\begin{lem}\label{lem_mov+ab}
Let $(X,\Delta)$ be a complex klt Calabi-Yau pair such that either $X$ is a surface, or $(X,\Delta) = (A,0)$ with $A$ an abelian variety. Then $\Mov^+(X) = \clMov^e(X)$.
\end{lem}
\begin{proof}
    If $(X,\Delta)$ is a klt Calabi-Yau pair of dimension two, then $\clMov(X) = \Nef(X)$, so it suffices to show that $\Nef^+(X) = \Nef^e(X)$. The inclusion $\Nef^+(X) \subset \Nef^e(X)$ follows from \cite[Corollary C]{lazic2020generalised}, and the reverse inclusion $\Nef^e(X) \subset \Nef^+(X)$ follows from \cite[Theorem 4.1]{totaro10conepairs}.
    
    Assume now that $(X,\Delta) = (A,0)$ with $A$ an abelian variety. Every effective divisor on $X$ is semiample by the theorem of the square \cite[II.6, Corollary 4]{mumford1970abelian}, hence $\clMov(X) = \overline{\Eff}(X)$. Moreover, by \cite[Lemma 1.1]{bauer1998cone}, we have
    \begin{equation}\label{eq_bauer}
        \Eff(X) \cap N^1(X)_\Q = \Nef(X)\cap N^1(X)_\Q,
    \end{equation}
    so $\Nef(X) = \overline{\Eff}(X) = \clMov(X)$. The equality $\Mov^+(X) = \clMov^e(X)$ then follows from the fact that $\Nef^+(X) = \Nef^e(X)$. Indeed, since every effective divisor on $X$ is nef, we have
    \[\Nef^e(X) = \Eff(X) = \operatorname{Conv}(\Eff(X) \cap N^1(X)_\Q) = \operatorname{Conv}(\Nef(X)\cap N^1(X)_\Q) = \Nef^+(X).\]
    This completes the proof.
\end{proof}
We now have all the tools we need to prove Theorem \ref{thm_A} and Corollary \ref{cor_B}.

\begin{proof}[Proof of Theorem \ref{thm_A}]
    Recall that the very general fibre $(X_s, \Delta \vert_{X_s})$ is isomorphic to $(W/G, \Delta_{W/G})$ where $W = A \times \prod_{i=1}^p Y_i \times \prod_{j=1}^q S_j$ is a product of 
    \begin{itemize}
        \item [(i)] an abelian variety $A$,
        \item [(ii)] projective primitive symplectic varieties $Y_i$ with $\Q$-factorial terminal singularities such that $b_2(Y_i) \geq 5$, $\Mov^+(Y_i) \subset \Eff(Y_i)$, and for every klt pair $(Y_i,D_i)$, any $(K_{Y_i} + D_i)$-MMP terminates with a good minimal model,
        \item [(iii)] smooth rational surfaces $S_j$ underlying klt Calabi–Yau pairs $(S_j,\Delta_j)$.
    \end{itemize}
    Moreover, $G$ is a finite subgroup of $\Aut(W,\sum_{j=1}^q p_j^*\Delta_j)$ acting freely in codimension one, and $\Delta_{W/G}$ is the unique $\mathbb{Q}$-divisor on $W/G$ such that
    \[K_W + \sum_{j=1}^q p_j^*\Delta_j = \gamma^*(K_{W/G}+\Delta_{W/G}),\]
    where $\gamma : W \to W/G$ is the quotient morphism. We now check that the assumptions of Theorem \ref{thm_C} are satisfied. From now on, we identify the pair $(X_s, \Delta \vert_{X_s})$ with the pair $(W/G, \Delta_{W/G})$. 
    
    \vspace{0.2cm}
    
    \emph{Step 1.} We prove that good minimal models exist for effective klt perturbations of $(W/G, \Delta_{W/G})$.
    
    \vspace{0.1cm}
    
    \noindent We first establish the corresponding statement for $(W,\Delta_W)$, and then descend it to the quotient. As recalled in the proof of Lemma \ref{lem_mov+ab}, every effective divisor on an abelian variety is nef and semiample. In particular, if $D_A$ is an effective $\R$-divisor on $A$ such that $(A,D_A)$ is klt, then $(A,D_A)$ is already a good minimal model. By assumption, any log-MMP for an effective klt pair on each $Y_i$ terminates with a good minimal model. Likewise, \cite[Theorem 1.1]{fujino2012minimal} shows that any log-MMP on an effective klt perturbation of $(S_j,\Delta_j + D_j)$ terminates with a good minimal model. Corollary \ref{cor_termprodgood} therefore implies that for any $\R$-divisor $D_W$ on $W$ such that $(W, \Delta_W + D_W)$ is klt and $\kappa(W, \Delta_W + D_W) \geq 0$,
    \begin{equation}\label{eq_terminategmm}
        \text{any}~ (K_W + \Delta_W + D_W)\text{-MMP terminates with a good minimal model of}~ (W,\Delta_W + D_W).
    \end{equation}
    We now pass to the quotient. Since $G$ is finite and acts freely in codimension one, the quotient morphism $\gamma : W \to W/G = X_s$ is a quasi-étale Galois cover. Let $D_{W/G}$ be an $\R$-divisor on $W/G$ such that the pair $(W/G,\Delta_{W/G} + D_{W/G})$ is klt and $\kappa(W/G, K_{W/G} + \Delta_{W/G} + D_{W/G}) \geq 0$. By \cite[Proposition 3.3 and 3.4]{denisi2024mmp}, any $(K_{W/G}+\Delta_{W/G} + D_{W/G})$-MMP terminates with a minimal model $((W/G)^{m},\Delta_{W/G}^{m} + D_{W/G}^{m})$ (although these statements are formulated for Galois quasi-étale covers between normal projective varieties with numerically trivial canonical divisor, the same arguments apply to Calabi-Yau pairs). More precisely, the above MMP lifts to a $G$-equivariant $(K_W+\Delta_W + D_W)$-MMP, where $D_W = \gamma^*D_{W/G}$, terminating with a minimal model $(W^{m,G},\Delta_W^{m,G} + D_W^{m,G})$ together with a quasi-étale Galois cover $\gamma_{m} : W^{m,G} \to (W/G)^{m}$ such that
    \begin{equation}\label{eq_pullbackequiv}
        K_{W^{m,G}}+\Delta_W^{m,G} + D_W^{m,G} \sim_{\R} \gamma_m^*\left(K_{(W/G)^{m}}+\Delta_{W/G}^{m} + D_{W/G}^m\right).
    \end{equation}
    On the other hand, this $G$-equivariant MMP is dominated by an ordinary $(K_W + \Delta_W + D_W)$-MMP which terminates with a good minimal model $(W^{m},\Delta_W^{m} + D_W^{m})$ by (\ref{eq_terminategmm}), and there exists a projective birational morphism $f_m : W^{m} \to W^{m,G}$ such that
    \begin{equation}\label{eq_pullback}
        K_{W^{m}} + \Delta_W^{m} + D_W^{m} \sim_{\R} f_m^*(K_{W^{m,G}}+\Delta_W^{m,G} + D_W^{m,G}).
    \end{equation}
    The situation is summarized by the following diagram:
    \begin{center}
        \begin{tikzcd}
            {(W^{m},\Delta_W^{m} + D_W^{m})} \arrow[rr, "f_m"] &  & {(W^{m,G},\Delta_W^{m,G} + D_W^{m,G})} \arrow[rr, "\gamma_{m}"] &  & {\left((W/G)^{m},\Delta_{W/G}^{m} + D_{W/G}^{m}\right)}
        \end{tikzcd}
    \end{center}
    Since $K_{W^{m}} + \Delta_W^{m} + D_W^{m}$ is semiample, \eqref{eq_pullback} and the fact that $f_m$ is projective and birational imply that $K_{W^{m,G}}+\Delta_W^{m,G}+D_W^{m,G}$ is semiample. Then \cite[Lemma 2.6]{das2022log}, together with (\ref{eq_pullbackequiv}), shows that $K_{(W/G)^{m}}+\Delta_{W/G}^{m} + D_{W/G}^m$ is semiample. Thus good minimal models exist for effective klt perturbations of $(W/G,\Delta_{W/G})$.
    
    \vspace{0.1cm}
    
    \emph{Step 2.} We show that the movable cone conjecture holds for $W/G$, that $\Mov^+(W/G) \subset \Eff(W/G)$ and that $\Mov^+(W/G)$ is well-clipped.
    
    \vspace{0.1cm}

    \noindent The movable cone conjecture holds for each factor of $W$ by \cite[Theorem 0.1]{Prendergast-Smith2012}, \cite[Theorem 1.2]{lehn2024morrison}, \cite[Theorem 4.1]{totaro10conepairs} and Lemma \ref{lem_mov+ab}. By Lemma \ref{lem_picsplit} and \cite[Lemma 2.16]{gachet2025well}, the closed movable cone of $W$ is the direct sum of the closed movable cones of its factors; the same holds for the modified cone of $W$. Hence \cite[Corollary 2.18]{gachet2025well} implies that the movable cone conjecture holds for $W$. Moreover, by \cite[Lemma 3.5]{gachet2025well} and Example \ref{ex_wellclipped}, 
    \begin{equation}\label{eq_wellclipped}
        \text{the cone}~ \clMov(W) ~\text{is well-clipped.}
    \end{equation}
    Therefore, \cite[Theorem 1.6]{gachet2025well} shows that the movable cone conjecture also holds for $W/G$.
    
    We now prove that $\Mov^+(W/G) \subset \Eff(W/G)$ and that $\Mov^+(W/G)$ is well-clipped. By Lemma \ref{lem_mov+ab} and by assumption, the modified movable cone of each factor of $W$ is contained in its effective cone. Therefore
    \begin{equation}\label{eq_movWeff}
        \Mov^+(W) = \Mov^+(A) \oplus \bigoplus_{i=1}^p \Mov^+(Y_i) \oplus \bigoplus_{j=1}^q \Mov^+(S_j) \subset \Eff(W).
    \end{equation}
    By \cite[Lemma 2.8]{denisi2024mmp}, the pullback by $\gamma$ induces an isomorphism $N^1(W/G)_\R \simeq N^1(W)^G_\R$, where $N^1(W)^G_\R$ denotes the $G$-invariant subspace. Identifying $N^1(W/G)_\R$ with a subspace of $N^1(W)_\R$, it follows from the proof of \cite[Lemma 2.8]{denisi2024mmp} that
    \begin{equation}\label{eq_effquot}
        \Eff(W/G) = \Eff(W) \cap N^1(W/G)_\R.
    \end{equation}
    On the other hand, \cite[Theorem 1.1]{gomez2025stable} gives $\Mov(W/G) = \Mov(W) \cap N^1(W/G)_\R$ (see also the proof of \cite[Theorem 1.6]{gachet2025well}), and Lemma \ref{lem_coneboundary} then implies that 
    \begin{equation}\label{eq_movquot}
        \clMov(W/G) = \clMov(W) \cap N^1(W/G)_\R.
    \end{equation}
    Combining (\ref{eq_movWeff}), (\ref{eq_effquot}), and (\ref{eq_movquot}), we conclude that $\Mov^+(W/G) \subset \Eff(W/G)$. Moreover, (\ref{eq_wellclipped}) together with \cite[Proposition 4.5]{gachet2025well} implies that $\Mov^+(W/G)$ is well-clipped.
    
    \vspace{0.1cm}
    
    \emph{Step 3.} We show that the movable cone conjecture holds for $X_\eta$ and that $\Mov^+(X_\eta) \subset \Eff(X_\eta)$.
    
    \vspace{0.1cm}
    
    \noindent Let $F = \C(S)$ be the function field of $S$ and let $\overline{F}$ be an algebraic closure of $F$. Write $\Delta = \sum_i a_i D_i$. Applying \cite[Lemma 2.1]{vial2013algebraic} to $f : X \to S$ and to each $f\vert_{D_i} : D_i \to S$, it follows that for a very general point $s \in S$, there exists a field isomorphism $\phi : \C \to \overline{F}$ and an isomorphism of schemes $X_{\overline{\eta}} \to X_s$ over $\Spec(\phi)$ such that $\alpha^*(\Delta_s) = \Delta_{\overline{\eta}}$. More precisely, for each $i$, we have the following Cartesian diagrams.
    \begin{equation}\label{eq_cartfib}
        \begin{tikzcd}
(D_i)_{\overline{\eta}} \arrow[dd] \arrow[rr, "\simeq"] &         & ( D_i)_s \arrow[dd] \\
                                                             & \square &                          \\
X_{\overline{\eta}} \arrow[dd] \arrow[rr, "\simeq"]          &         & X_s \arrow[dd]           \\
                                                             & \square &                          \\
\operatorname{Spec}(\overline{F}) \arrow[rr, "\simeq"]       &         & \operatorname{Spec}(\C) 
\end{tikzcd}
    \end{equation}
    In particular, $\alpha$ induces an isomorphism between $\psaut(X_s, \Delta_s)$ and $\psaut(X_{\overline{\eta}}, \Delta_{\overline{\eta}})$. We deduce by Step 2 that $\clMov(X_{\overline{\eta}})$ is well-clipped, that the movable cone conjecture holds for $(X_{\overline{\eta}}, \Delta_{\overline{\eta}})$ and that
    \begin{equation}\label{eq_mov+effggen}
        \Mov^+(X_{\overline{\eta}}) \subset \Eff(X_{\overline{\eta}}).
    \end{equation}
    The movable cone conjecture for $(X_{\eta}, \Delta_{\eta})$ then follows from \cite[Theorem 1.6]{gachet2025well}. By Lemma \ref{lem_lins} and (\ref{eq_mov+effggen}), we also obtain $\Mov^+(X_{\eta}) \subset \Eff(X_{\eta})$.
    
    \vspace{0.1cm}
    
    \emph{Step 4.} We apply Theorem \ref{thm_C}.
    
    \vspace{0.1cm}
    
    \noindent The above discussion shows that there exists there exists a rational polyhedral cone $P_\eta \subset \Mov^+(X_\eta) \subset \Eff(X_\eta)$ such that
    \[\Mov(X_\eta) \subset \Mov^+(X_\eta) = \psaut(X_\eta, \Delta_\eta) \cdot P_\eta.\]
    The conclusion now follows from Theorem \ref{thm_C}.
\end{proof}

\begin{proof}[Proof of Corollary \ref{cor_B}]
    Note that the very general fibre $(X_s, \Delta_s)$ is of the form $(W/G,\Delta_{W/G})$, where
    \[W = A \times \prod_{i=1}^p Y_i \times \prod_{k=1}^\ell Y'_k \times \prod_{j=1}^q S_j,\]
    each $Y'_k$ is a projective irreducible holomorphic symplectic manifold of a known type, and
    \[G = \{\id_A\} \times \prod_{i=1}^p \{\id_{Y_i}\} \times \prod_{k=1}^\ell \pi_1(Z_k) \times \prod_{j=1}^q \{\id_{S_j}\}\]
    is a finite subgroup of $\Aut(W, \sum_{j=1}^q p_j^*\Delta_j)$. We now verify that the assumptions of Theorem \ref{thm_A} are satisfied. The only points requiring justification are:
    \begin{itemize}
        \item [(i)] the inclusions $\Mov^+(Y_i) \subset \Eff(Y_i)$ and $\Mov^+(Y'_k) \subset \Eff(Y'_k)$;
        \item [(ii)] the existence of good minimal models for effective klt pairs on each $Y_i$ and $Y'_k$.
    \end{itemize}
    Item (i) follows from \cite[Lemma 2.17]{denisi2022pseudo}, and item (ii) follows from the termination of flips on a projective irreducible holomorphic symplectic manifold by \cite[Theorem 1.2]{lehn2016deformations}, together with the SYZ conjecture, which is known for all known types by \cite[Theorem 1.5]{bayer2014mmp}, \cite[Theorem 1.3]{markman2014lagrangian}, \cite[Proposition 3.38]{yoshioka16bridgeland}, \cite[Theorem 7.2, Corollary 7.3]{mongardi2021monodromy}, \cite[Theorem 2.2]{mongardi2022birational} and \cite[Corollary 1.1]{matsushita17isotropic}. 

    More precisely, let $V$ be a projective irreducible holomorphic symplectic manifold, and let $D$ be an effective $\R$-divisor on $V$ such that $(V,D)$ is klt. Then any $(K_V+D)$-MMP on $(V,D)$ terminates with a minimal model $(V^{\min}, D^{\min})$, and each variety $V_i$ appearing at the $i$-th step admits a crepant resolution by an irreducible holomorphic symplectic manifold $T_i$ by \cite[Lemma 4.1]{lehn2016deformations}. We thus obtain the following diagram:
    \begin{center}
        \begin{tikzcd}
V = T_0 \arrow[dd, equal] \arrow[rr, dashed] &  & T_1 \arrow[rr, dashed] \arrow[dd] &  & T_2 \arrow[rr, dashed] \arrow[dd] &  & \dotsc \arrow[rr, dashed] \arrow[dd] &  & T_n \arrow[dd]                \\
                                                             &  &                                   &  &                                   &  &                                      &  &                               \\
V = W_0 \arrow[rr, dashed]                                   &  & V_1 \arrow[rr, dashed]            &  & V_2 \arrow[rr, dashed]            &  & \dotsc \arrow[rr, dashed]            &  & V_n =: V^{\min}
\end{tikzcd}
    \end{center}
    In particular, $V$ is birational to $T_n$, and hence they are deformation equivalent by \cite[Theorem 4.6]{huybrechts1999compact}. Consequently, if $V$ is of a known type, then so is $T_n$, and the SYZ conjecture holds for $T_n$. It follows that $(V^{\min},D^{\min})$ is a good minimal model of $(V,D)$, which completes the proof.
\end{proof}

\printbibliography
\Addresses

\end{document}